\documentclass[11pt]{amsart}

\newtheorem{theorem}{Theorem}[section]

\newtheorem{corollary}[theorem]{Corollary}

\newtheorem{lemma}[theorem]{Lemma}

\newtheorem{proposition}[theorem]{Proposition}

\theoremstyle{definition}

\usepackage[colorlinks, bookmarks=true]{hyperref}
\usepackage{color,graphicx,shortvrb}

\usepackage{enumerate}
\usepackage{amsmath}
\usepackage{amssymb}

\usepackage[latin 1]{inputenc}

\def\J#1#2#3{ \left\{ #1,#2,#3 \right\} }

\def\11{\textbf{$1$}}

\newcommand{\eps}{\varepsilon}
\newcommand{\norm}[1]{\|#1\|}
\newcommand{\Norm}[1]{\Bigl\|#1\Bigr\|}
\newcommand{\bgl}{\begin{eqnarray}}
\newcommand{\bglst}{\begin{eqnarray*}}
\newcommand{\egl}{\end{eqnarray}}
\newcommand{\eglst}{\end{eqnarray*}}
\newcommand{\N}{\rm{I\!N}}
\newcommand{\R}{\rm{I\!R}}

\newcommand{\beinschub}{\vspace{2mm}\begin{small}\newline\makebox[5ex]{}
     \begin{minipage}[t]{75ex}}
\newcommand{\eeinschub}{\end{minipage}\end{small}\vspace{3mm}\newline}
\newcommand{\betr}[1]{| #1  |}

\begin{document}

\numberwithin{equation}{section}

\title[Perturbation of $\ell_1$-copies in preduals of JBW$^*$-triples]{Perturbation of $\ell_1$-copies in preduals of JBW$^*$-triples}

\author[A.M. Peralta]{Antonio M. Peralta}
\address{Departamento de An{\'a}lisis Matem{\'a}tico, Universidad de Granada,\\
Facultad de Ciencias 18071, Granada, Spain}
\curraddr{Visiting Professor at Department of Mathematics, College of Science, King Saud University, P.O.Box 2455-5, Riyadh-11451, Kingdom of Saudi Arabia.}
\email{aperalta@ugr.es}

\author[H. Pfitzner]{Hermann Pfitzner}
\address{Université d'Orléans,\\
BP 6759,\\
F-45067 Orléans Cedex 2,\\
France}
\email{pfitzner@labomath.univ-orleans.fr}

\thanks{First author partially supported by the Spanish Ministry of Science and Innovation,
D.G.I. project no. MTM2011-23843, Junta de Andaluc\'{\i}a grant FQM375 and Deanship of Scientific
Research at King Saud University (Saudi Arabia) research group no. RGP-361.}



\date{}
\maketitle

\begin{abstract}
Two normal functionals on a JBW$^*$-triple are known to be orthogonal if and only if they are $L$-orthogonal (meaning that they
span an isometric copy of $\ell_1(2)$).
This is shown to be stable under small norm perturbations in the following sense: if the linear span of the two functionals
is isometric up to $\delta>0$ to $\ell_1(2)$, then the functionals are less far (in norm) than $\eps>0$ from two orthogonal functionals,
where $\eps\to0$ as $\delta\to0$.
Analogous statements for finitely and even infinitely many functionals hold as well.
And so does a corresponding statement for non-normal functionals.
Our results have been known for C$^*$-algebras.
\end{abstract}

\maketitle
\thispagestyle{empty}

\section{Introduction}\label{sec:intro}\noindent
The starting point of this note consists in two well-known facts. First, two elements in the predual of a JBW$^*$-triple
are orthogonal if and only if they span the two-dimensional $\ell_1(2)$ isometrically and second, in preduals of von Neumann algebras
this still makes sense after small norm perturbations, moreover not only for two but for finitely and, up to subsequences,
even infinitely many elements.\smallskip

For example, if a sequence $(\varphi_n)$ in $L^1([0,1])$ is such that $$\sum\betr{\alpha_n}\geq\norm{\sum\alpha_n\varphi_n}\geq r\sum\betr{\alpha_n},$$
then there are pairwise orthogonal $\widetilde\varphi_n$ such that $\norm{\varphi_n-\widetilde\varphi_n}<\eps$, with $\eps\to0$ as $r\to1$
(and, of course, with $\norm{\sum\alpha_n\widetilde\varphi_n}=\sum\betr{\alpha_n}$), see \cite{Dor}.
Briefly, in $L^1$ a sequence near to an isometric copy of $\ell_1$ is near to an orthogonal sequence.
Up to subsequences the same follows from \cite[Th.\ 1.2]{Pfi02} for arbitrary von Neumann preduals.
Analogous non-normal versions hold, too: it can be deduced from \cite[Prop.\ 1.3]{Pfi02} that if a sequence $(\varphi_n)$
in the dual of a C$^*$-algebra $A$ is as above then for any $\eps>0$ there are pairwise orthogonal elements $c_k\in A$ such that
$\varphi_{n_k}(c_k)>(1-\eps)r$ for some subsequence $(\varphi_{n_k})$; a similar formulation
(reminiscent of Pe{\l}czy\'{n}ski's property $(V)$ or Grothendieck's criterion of weak compactness in the dual of a $C(K)$-space)
is that if one accepts to replace $r$ by a worse constant (e.g. $r^2/2$) in the last inequality, then the $c_k$'s can be considered
to be selfadjoint elements of a commutative subalgebra of $A$ \cite[\S6, Lem.\ 6.3]{Pfi02}.\smallskip

In view of the numerous generalizations of geometric (=Banach space theoretic) properties from C$^*$-algebras to JB$^*$-triples it is natural
to conjecture similar results for JB$^*$-triples. The aim of this article is to state and prove them.\smallskip

Let us describe these results. They divide into two parts, contained in Sections \ref{sec 3} and \ref{sec 4}, respectively,
depending on whether the $\ell_1$-copies are of finite or infinite dimension.
Basic to all this, as already alluded to in the first paragraph, is the result of Y. Friedman, B. Russo \cite[Lem.\  2.3]{FriRu87},
according to which
two functionals are algebraically orthogonal if and only if they are $L$-orthogonal where the latter means that the two functionals span
an isometric  copy of $\ell_1(2)$ (see Section \ref{sec preliminaries} for definitions).\smallskip

The main result of Section \ref{sec 3}, Theorem \ref{t quantitative version of L-orthogonality in the dual space of a JB*-triple n-functionals},
yields a quantification of algebraic orthogonality for finitely many arbitrary elements in the dual of a JB$^*$-triple $E$:
if functionals $\varphi_1,\ldots,\varphi_n$ in $E^*$ span $\ell_1(n)$ $(1-\delta)$-isomorphically then they are near to pairwise orthogonal functionals
as $\delta$ is near to $0$ and moreover they attain their norm up to a given $\eps>0$ on some pairwise orthogonal elements in $E$.
A quantification of orthogonality in $E$ is not possible in general but it is for tripotents, see Proposition
\ref{p quantitative version of M-orthogonal tripotents}.\smallskip

Section \ref{sec 4} contains what has been described in the second paragraph.
More specifically, if a  bounded sequence $(\varphi_n)$ in a JBW$^*$-predual $W_*$ spans $\ell_1$ almost isometrically,
then according to Theorem \ref{t asymptotically isometric copies of ell_1 JBW*-triple preduals} there are pairwise orthogonal
$\widetilde\varphi_k$ such that $\norm{\varphi_{n_k}-\widetilde\varphi_k}\to0$ for some subsequence $\varphi_{n_k}$.
The non-normal case, treated in Theorem \ref{t isomorphic copies in the dual of a JB-triple}, can be resumed by saying that if the $\varphi_n$'s
span $\ell_1$ $r$-isomorphically in the dual of a JB$^*$-triple $E$, then $E$ contains an abelian subtriple such that the
restrictions of an appropriate subsequence of the $\varphi_n$ to this subtriple still span $\ell_1$ $(1-\eps)$-isomorphically for any given $\eps>0$.
A quantitative version of Theorem \ref{t isomorphic copies in the dual of a JB-triple} is already contained in \cite[Th.\ 2.3]{FerPe09}
and, what is more, the arguments in \cite{FerPe09} and \cite{ChuMellon97} seem to lend themselves to a quantification that gives our Theorem
\ref{t isomorphic copies in the dual of a JB-triple}.
We refrained from pursuing this approach for it seems more natural to deduce the infinite dimensional case from the finite dimensional one,
all the more because the latter has an interest in its own.\bigskip

\section{Preliminaries}\label{sec preliminaries}
\noindent
We shall follow the standard notation employed, for example in \cite{FerPe06}, \cite{FerPe09} or \cite{BunPe07}.
For Banach space theory we refer, e.g., to \cite{Diestel1984, FHHMZ, JohLin_Vol1-2}.\smallskip

We recall that a JB$^*$-triple \cite{Kaup83} is a complex Banach space $E$ equipped with a
continuous ternary product $\{.,.,.\}$ symmetric and bilinear in
the outer variables and conjugate linear in the middle one satisfying \begin{equation}\label{eq Jordan identity}
L(x,y) \J abc = \J {L(x,y)a}bc - \J a{L(y,x)b}c + \J
ab{L(x,y)c},\end{equation} such that $\|L(a,a)\| = \|a\|^2$ and $L(a,a)$ is an
hermitian operator on $E$ with non-negative spectrum, where
$L(a,b)$ is given by $L(a,b) y =\J aby$.\smallskip

Every C$^*$-algebra is a JB$^*$-triple with respect to the triple
product given by \hyphenation{product} $ \J xyz =\frac12 (x y^* z +z y^*
x).$ The same triple product equipes the space $B(H,K)$, of all bounded linear operators between
complex Hilbert spaces $H$ and $K$, with a structure of JB$^*$-triples. Among the examples involving Jordan algebras,
we can say that every JB$^*$-algebra is a JB$^*$-triple under the triple
product $\J xyz = (x\circ y^*) \circ z + (z\circ y^*)\circ x - (x\circ z)\circ y^*$.\smallskip

An element $u$ in a JB$^*$-triple $E$ is said to be a
\emph{tripotent} when it is a fixed point of the triple product, that is, when $u= \J uuu$.
Given a tripotent $u\in E$, the
mappings $P_{i} (u) : E\to E_{i} (u)$, $(i=0,1,2)$, defined by
$$P_2 (u) = L(u,u)(2 L(u,u) -id_{E}), \ P_1 (u) = 4
L(u,u)(id_{E}-L(u,u)),$$ $$ \ \hbox{ and } P_0 (u) =
(id_{E}-L(u,u)) (id_{E}-2 L(u,u)),$$ are contractive linear
projections, called the \emph{Peirce projections} associated with $u$.
The range of $P_i(u)$ is the eigenspace $E_i(u)$ of $L(u, u)$ corresponding to the eigenvalue $\frac{i}{2},$ and
$$E= E_{2} (u) \oplus E_{1} (u)\oplus E_0 (u)$$ is the \emph{Peirce decomposition} of $E$ relative to $u$.
Furthermore, the following Peirce rules are
satisfied, \begin{equation}
\label{eq peirce rules1} \J {E_{2} (u)}{E_{0}(u)}{E} = \J {E_{0}
(u)}{E_{2}(u)}{E} =\{0\},
\end{equation} \begin{equation}
\label{eq peirce rules2} \J {E_{i}(u)}{E_{j} (u)}{E_{k}
(u)}\subseteq E_{i-j+k} (u),
\end{equation} where $E_{i-j+k} (u)=\{0\}$ whenever
$i-j+k \notin \{ 0,1,2\}$ (\cite{FriRu85} or \cite[Th.\ 1.2.44]{Chu2012}).
For $x,y,z$ in a JB$^*$-triple $E$ we have \cite[Cor.\ 3]{FriRu86}
\bgl \|\{x,y,z\}\|\leq \norm{x}\norm{y}\norm{z}.\label{eq 1a}\egl
A tripotent $u$ is called \emph{complete} if $E_0 (u) $ reduces to $\{0\}$.\smallskip

The Peirce-2 subspace $E_{2} (u)$ is a unital JB$^*$-algebra
with unit $u$, product $a \circ_{u} b = \J aub$ and involution $a^{\sharp_{u}} = \J uau$
(c.f. \cite[Theorem 2.2]{BraKaUp78} and \cite[Theorem 3.7]{KaUp77}; \cite[p.\ 185]{Chu2012}).\smallskip

A JBW$^*$-triple is a JB$^*$-triple which is also a dual Banach space. Every JBW$^*$-triple admits a unique isometric predual
and its triple product is separately weak$^*$-continuous (\cite{BarTi}, \cite{Horn87}, \cite[Th.\ 3.3.9]{Chu2012}). Consequently, the Peirce projections
associated with a tripotent in a JBW$^*$-triple are weak$^*$-continuous.
The second dual of a JB$^*$-triple is a JBW$^*$-triple such that its triple product reduces to the original one (cf.\ \cite{Dineen86}, \cite[Cor.\ 3.3.5]{Chu2012}).
The class of JBW$^*$-triples includes all von Neumann algebras.
Functionals on a JBW$^*$-triple $W$ are called normal if they belong to the predual $W_*$.\smallskip

JBW$^*$-triples play, in the category of JB$^*$-triples, a similar role to that played by von Neumann algebras in the setting of C$^*$-algebras.
A JB$^*$-triple need not have any non-zero tripotent. However, since the complete tripotents of a JB$^*$-triple $E$ coincide with the complex
and the real extreme points of its closed unit ball (cf. \cite[Lem.\ 4.1]{BraKaUp78}, \cite[Prop.\ 3.5]{KaUp77}, \cite[Th.\ 3.2.3]{Chu2012}),
the Krein-Milman theorem implies that every JBW$^*$-triple contains an abundant set of tripotents.\smallskip

Given elements $a, b$ in a JB$^*$-triple $E$, the symbol $Q(a,b)$ will denote the
conjugate linear operator on $E$ defined by $Q(a,b)(x):=\J axb$. We write $Q(a)$ instead of $Q(a,a)$.
The \emph{Bergmann operator} $B(a,b): E \to E$ is the mapping given by
$B(a,b) (z) = z -2L(a,b) (z) + Q(a)Q(b) (z),$ for all $z$ in $E$
(compare \cite{Loos} or \cite[page 305]{Up85}). In the particular
case of $u$ being a tripotent, we have $P_0 (u) =
B(u,u)$.\medskip

Throughout the paper, given a Banach space $X$, we consider $X$ as a closed subspace of $X^{**},$ via its natural isometric embedding,
and for each closed subspace $Y$ of $X$ we shall identify $\overline{Y}^{\sigma(X^{**},X^{*})}$, the weak$^*$-closure of $Y$ in $X^{**}$,
with $Y^{**}$.\smallskip

A normalized sequence $(x_n)$ in a Banach space $X$ is said to span $\ell_1$ $r$-isomorphically if
$\|\sum_n\alpha_n x_n\|\geq r\sum_n\betr{\alpha_n}$
for all scalars $\alpha_n$. If there is a sequence $(\delta_m)$ such that
$0\leq\delta_m\to0$ and $(x_n)_{n\geq m}$ spans $\ell_1$ $\delta_m$-isomorphically for all $m$ then $(x_n)$ is said to span $\ell_1$ almost isometrically.\bigskip

\noindent{\sc The strong-$^*$-topology}.
Given a norm-one element $\varphi$ in the predual $W_*$ of a JBW$^*$-triple $W$, and
a norm-one element $z$ in $W$ with $\varphi (z) =1$, it follows from \cite[Proposition
1.2]{BarFri90} that the assignment $$(x,y)\mapsto \varphi\J xyz$$ defines a positive
sesquilinear form on $W.$ Moreover, for every norm-one element $w$ in $W$
satisfying $\varphi (w) =1$, we have $\varphi\J xyz = \varphi\J
xyw,$ for all $x,y\in W$. The mapping $ x\mapsto \|x\|_{\varphi}:=
\left(\varphi\J xxz\right)^{\frac{1}{2}},$ defines a prehilbertian
seminorm on $W$.
The \emph{strong$^*$-topology} of $W$, introduced by T.J.
Barton and Y. Friedman in \cite{BarFri90}, is
the topology on $W$ generated by the family $\{
\|\cdot\|_{\varphi}:\varphi\in {W_*}, \|\varphi \| =1
\},$ and will be denoted by $s^*(W,W_*)$.
From \cite[page 258]{BarFri90} we get $|\varphi (x)| \leq \|x\|_{\varphi}$ for any $x\in W,$ and it is clear from this that $s^*(W,W_*)$
is stronger than the weak$^*$-topology of $W$.\smallskip

It is known that the triple product of a JBW$^*$-triple is jointly strong$^*$-continuous on bounded sets
(\cite[Th.\ 9]{PeRo}, \cite[Th., page 103]{Rodr1991}).
Another interesting property tells us that the strong*-topology of a JBW$^*$-triple $W$ is compatible with the duality $(W, W_{*})$
(i.e. a linear functional on $W$ is weak$^*$-continuous if, and only if, it is strong$^*$-continuous, see \cite[Th.\ 9]{PeRo}).
The bipolar theorem implies that for convex sets of $W$, weak$^*$-closure and strong$^*$-closure coincide.
It follows that the closed unit ball of a weak$^*$-dense JB$^*$-subtriple $E$ of a JBW$^*$-triple $W$ is strong$^*$-dense in the closed unit ball
of $W$. This result, known as \emph{Kaplansky Density theorem for JBW$^*$-triples}, was established by J.T. Barton and Y. Friedman in \cite[Cor.\ 3.3]{BarFri90}.\smallskip

In 2001, L.J. Bunce culminated the description of the fundamental properties of the strong$^*$-topology showing that
for every JBW$^*$-subtriple $F$ of a JBW$^*$-triple $W$, the strong$^*$-topology of $F$ coincides with the restriction to $F$ of
the strong$^*$-topology of $W$, that is, $s^*(F,F_*) = s^* (W,W_*)|_{F}$ \cite{Bunce01}. It is also known that a linear map between
JBW$^*$-triples is strong$^*$-continuous if, and only if, it is weak$^*$-continuous (compare \cite[page 621]{PeRo}). \bigskip

\noindent{\sc Functional calculus, open and closed tripotents}.
Let $x$ be an element in a JB$^*$-triple $E$. Throughout the paper,
the symbol $E_{x}$ will stand for the norm-closed subtriple of $E$
generated by $x$. It is known that $E_{x}$ is JB$^*$-triple
isomorphic to the abelian C$^*$-algebra $C_{0} (L)$ of all complex-valued
continuous functions on $L$ vanishing at $0$, where $L$ is a
locally compact subset of $(0,\|x\|]$ satisfying that $L\cup\{0\}$
is compact. Further, there exists a JB$^*$-triple isomorphism $\Psi : E_x
\to C_{0} (L)$ satisfying $\Psi (x) (t) = t,$ for all $t$ in
$L$ (compare \cite[1.15]{Kaup83}).
Given a continuous complex-valued function $f: L\cup\{0\} \to \mathbb{C}$ vanishing at 0,
the \emph{continuous triple functional calculus} $f(x)$ will have its usual meaning $f(x)=\Psi^{-1}(f)$.\smallskip

We define $x^{[1]}:=x$ and $x^{[2n+1]} = \{ x, x, x^{[2 n-1]}\}$ for every $n\in\mathbb{N}$. JB$^*$-triples are power associative, that is, $$\J{x^{[2k-1]}}{x^{[2l-1]}}{x^{[2 m-1]}}=x^{[2(k+l+m)-3]},$$ for every $k,l,m \in \mathbb{N}$
(cf. \cite[\S 3.3]{Loos} or \cite[Lem.\ 1.2.10]{Chu2012} or simply apply the Jordan identity).\smallskip

Suppose now that $\norm{x}=1$ and that $E$ is a subtriple of a JBW$^*$-triple $W$, for example $W=E^{**}$.
It is known that $(x^{[2n+1]})$ converges in the strong$^*$-topology to the tripotent $u(x)=\chi_{\{1\}}(x)\in\overline{E_x}^{w^*}\subset W$,
which is called the \emph{support tripotent} of $x$ (\cite[Lem.\ 3.3]{EdRu96}). (By $\chi_A$ we denote the characteristic function of a set $A$.)
By functional calculus there exist, for each $n\in\mathbb{N}$, unique elements $x^{[\frac{1}{2n-1}]}$ in $E_x \cong C_{0} (L)$ satisfying
$\left(x^{[\frac{1}{2n-1}]}\right)^{[2n-1]} = x$. The latter are strong$^*$-convergent to the tripotent
$r(x)=\chi_{(0,1]}(x)\in\overline{E_x}^{w^*}\subset W$, which is called the \emph{range tripotent of $x$}.
The tripotent $r(x)$ is the smallest tripotent $e\in W$ satisfying that $x$ is
positive in the JBW$^*$-algebra $W_{2} (e)$ (see, for example, \cite[comments before Lemma 3.1]{EdRu88} or \cite[\S 2]{BuChuZa}).
The inequalities $$u(x)\leq x^{[2n+1]} \leq x\leq r(x)$$
hold in $W_2(r(x))$ for every norm-one element $x\in E.$\smallskip

A tripotent $u$, in a JB$^*$-triple $E$, is said to be \emph{bounded} if there exists a norm-one
element $x\in E$ such that $L(u,u)x=u$. The element $x$ is called a bound of $u$
and we write $u\leq x$. We shall write $y\leq u$ whenever $y$ is a positive element in
the JB$^*$-algebra $E_2 (u)$ (cf.\ \cite[pages 79-80]{FerPe06}).
A JB$^*$-triple $E$ need not have a cone of positive elements and the lacking of order implies that the symbol $x\leq y$ makes no sense
for general elements $x,y\in E$. However, it should be remarked that, given $x,y\in E$ and tripotents $u,v\in E$ with $x\leq u$
and $u\leq y\leq v$, we have $x\leq u\leq y \leq v$ with respect to the natural order of the JB$^*$-algebra $E_2 (v)$.
Note further, for Lemma \ref{l FerPer MathScan 2009 L22} below, that we have $u\leq x$ in the just mentioned sense if we have
$u\leq x$ in the JBW$^*$-algebra $E_2^{**}(r(x))$.\smallskip 

Inspired by the notion of open projection in the bidual of a C$^*$-algebra introduced and studied by C.\ Akemann, L.\ Brown, and G.K.\ Pedersen
(cf. \cite{Akemann69} or \cite{AkPed73,AkPed92} or \cite[Proposition 3.11.9]{Ped}), C.M.\ Edwards and G.T.\ R\"uttimann develop the notion of
open tripotent in the bidual of a JB$^*$-triple $E$: we say that a tripotent $e$ in $E^{**}$ is \emph{open} if $E^{**}_{2} (e) \cap E$ is
weak$^*$-dense in $E^{**}_{2} (e)$ (see \cite[page 167]{EdRu96}). It is known that the range tripotent of a norm-one element of $E$ is open
(cf. \cite[Proposition 2.1]{BuChuZa}).
A tripotent $e$ in $E^{**}$ is said to be \emph{compact-$G_{\delta}$} (relative to $E$) if there exists a norm-one element $x$ in $E$ such that
$e$ coincides with $u(x)$, the support tripotent of $x$. A tripotent $e$ in $E^{**}$ is said to be \emph{compact} (relative to $E$)
if there exists a decreasing net $(e_{\lambda})$ of tripotents in $E^{**}$ which are compact-$G_{\delta}$ with infimum $e$,
or if $e$ is zero (cf. \cite[pages 163-164]{EdRu96}). In the terminology of \cite{FerPe06}, we say that a tripotent $u$ in $E^{**}$ is closed
if $E\cap E_0^{**}(u)$ is weak$^*$-dense in $E_0^{**}(u)$. The equivalence established in \cite[Th.\ 2.6]{FerPe06} shows that a tripotent
$e\in E^{**}$ is compact if and only if $e$ is closed and bounded by an element of $E$.\medskip

\noindent
{\sc Small perturbation of a normal functional.}
Let $\varphi\in W_*$ be a functional in the predual of a JBW$^*$-triple $W$ and let $e$ be a tripotent in $W$.
In \cite[Proposition 1]{FriRu85}, Y. Friedman and B. Russo prove that $\left\| \varphi P_{2}( e) \right\|
=\left\| \varphi \right\|$ if and only if $\varphi =\varphi P_{2} (e).$ Using the techniques of ultraproducts of Banach spaces,
J. Becerra-Guerrero and A. Rodríguez Palacios obtained the following quantitative version of the above property which will be used
throughout this article.

\begin{lemma}\label{l BeRo Lemma 2.2}\cite[Lem.\ 2.2]{BeRo03} Given $\varepsilon > 0$, there exists $\eta = \eta(\varepsilon)> 0$ such that,
for every JB$^*$-triple $E$, every non-zero tripotent $e$ in $E$, and every $\varphi$ in $E^*$ with $\left\|\varphi\right\|\leq 1$ and
$\left\| \varphi P_2 (e) \right\| \geq 1 -\eta$, we have  $\left\|\varphi - \varphi P_2 (e)\right\| < \varepsilon.$ $\hfill\Box$
\end{lemma}\medskip

\noindent{\sc Orthogonality and geometric $M$- and $L$-orthogonality}.
We recall that elements $a,b$ in a JB$^*$-triple $E$ are said to be \emph{algebraically orthogonal} or simply \emph{orthogonal} (written $a\perp b$) if $L(a,b) =0$.
If we consider a C$^*$-algebra as a JB$^*$-triple then two elements are orthogonal in the C$^*$-sense if and only if they are orthogonal
in the triple sense.
It is known (compare \cite[Lem.\ 1]{BurFerGarMarPe}) that $a\perp b$ if and only if one of the following statements holds:\smallskip

\begin{tabular}{ccc}\label{ref orthogo}
  
  $\J aab =0;$ & $a \perp r(b);$ & $r(a) \perp r(b);$ \\
  & & \\
  $E^{**}_2(r(a)) \perp E^{**}_2(r(b))$; & $r(a) \in E^{**}_0 (r(b))$; & $a \in E^{**}_0 (r(b))$; \\
  & & \\
  $b \in E^{**}_0 (r(a))$; & $E_a \perp E_b$. & \\

\end{tabular}\medskip

It follows from Peirce rules (\ref{eq peirce rules1}) that,
for each tripotent $u$ in a JB$^*$-triple $E$, $E_0(u) \perp E_2 (u).$\smallskip

For each norm one functional $\varphi$ in the predual of a JBW$^*$-triple $W$, the square of the prehilbertian seminorm $\norm{\cdot}_{\varphi}$ is additive on orthogonal elements:
$$\norm{a+b}_{\varphi}^2=\norm{a}_{\varphi}^2+\norm{b}_{\varphi}^2,\quad\forall a\perp b.$$

\noindent
We recall that a functional $\phi$ in the predual of a JBW$^*$-algebra $M$ is said to be \emph{faithful} if for each $a\geq 0$ in $M$, $\phi (a) =0$ implies $a=0$.\smallskip

Let $\varphi$ be a norm-one functional in the predual of a
JBW$^*$-triple $W$. By \cite[Prop.\ 2]{FriRu85}, there exists a
unique tripotent $e= e(\varphi) \in W$ satisfying $\varphi =
\varphi P_{2} (e)$ and $\varphi|_{W_{2}(e)}$ is a faithful normal
state of the JBW$^*$-algebra $W_{2} (e)$. This unique tripotent $e$ is called the \emph{support tripotent} of
$\varphi$, and will be denoted by $e(\varphi)$. (Note that at the time of the writing of \cite{FriRu85} condition \cite[(1.13)]{FriRu85}
was not yet known to hold for all JBW$^*$-triples.)\smallskip

Now, according to \cite{FriRu87} and \cite{EdRu01}, we define two functionals $\varphi$ and $\psi$ in the predual of a JBW$^*$-triple $W$ to be
\emph{algebraically orthogonal} or simply \emph{orthogonal}, denoted by $\varphi\perp \psi$, if their support tripotents are orthogonal in $W$, that is $e(\varphi)\perp e(\psi)$.\medskip

Elements $x,y$ in a normed space $X$ are said to be \emph{$L$-orthogonal}
(and we write $x \perp_{L} y$) if $\|x \pm y\| = \|x\| + \|y\|$, and are said to be \emph{$M$-orthogonal} (denoted by $x \perp_M y$) if $\|x \pm y\| = \max\left\{ \|x\| , \|y\|\right\}$.\smallskip

Given $a,b$ in $E$, it follows from \cite[Lem.\  1.3$(a)$]{FriRu85} that $a\perp_{M} b$ whenever $a\perp b$.
In general the reverse implication does not hold, for example $(1/2,1,0)$ and $(1/2,0,1)$ in the C$^*$-algebra $l^{\infty}(3)$ are
$M$-orthogonal but not orthogonal.
The following result is borrowed from \cite{FriRu87} and \cite{EdRu01}.

\begin{lemma}\label{l EdRu01 L-orthogonal}\cite[Lem.\  2.3]{FriRu87} and \cite[Theorem 5.4 and Lemma 5.5]{EdRu01}
Let $\varphi$ and $\psi$ be two functionals in the predual of a JBW$^*$-triple $W$. Then $\varphi\perp \psi$ if, and only if,
$\varphi\perp_{L} \psi$. Furthermore, given two tripotents $e$ and $u$ in $W$, then $e\perp u$ if, and only if,
$e\perp_{M} u$. $\hfill\Box$\end{lemma}
\noindent
Note that $\varphi\perp_{L} \psi$ in the lemma (with $\|\psi\|,\|\varphi\|\neq 0$) is equivalent to
$$\norm{\alpha\frac{\varphi}{\norm{\varphi}}+\beta\frac{\psi}{\|\psi\|}}=\betr{\alpha}+\betr{\beta},$$ for any $\alpha,\beta\in\mathbb{C}$.

\section{Quantitative versions of $M$- and $L$-orthogonality in JBW$^*$triples and their predual spaces}\label{sec 3}

\noindent
The main goal of this section is to establish quantitative versions of Lemma \ref{l EdRu01 L-orthogonal}
(see Propositions \ref{p quantitative version of L-orthogonality in the dual space of a JB*-triple},
\ref{p quantitative version of M-orthogonal tripotents} and Theorem
\ref{t quantitative version of L-orthogonality in the dual space of a JB*-triple n-functionals} below).
The proof will follow from a series of technical results.
The next two lemmas, which are included here for the sake of completeness, are borrowed from \cite{FerPe09}.

\begin{lemma}\label{l FerPer MathScan 2009 L.1.2}\cite[Lem.\  1.2]{FerPe09}
Let $E$ be a JB$^*$-triple, $e$ a tripotent in $E,$ and $x$ a
norm-one element in $E$ with $e \leq x$. Then $B(x,x)$ is a
contractive operator and $B(x,x) (y)$ belongs to $E_{0} (e),$ for
every $y$ in $E$. $\hfill\Box$
\end{lemma}

\begin{lemma}\label{l FerPer MathScan 2009 L21}\cite[Lem.\  2.1]{FerPe09}
Let $E$ be a JB$^*$-triple, $v$ be a tripotent in $E,$ and
$\varphi$ an element in the closed unit ball of $E^{*}$. Then for each
$y\in E_2(v)$ with $\|y\|\leq 1$ we have \begin{eqnarray} |\varphi (x-B(y,y)
(x))| \leq 21 \| x\| \|v\|_{\varphi},\end{eqnarray} for every $x\in E$. $\hfill\Box$\end{lemma}

\noindent
We shall also need an appropriate version of \cite[Lem.\  2.2]{FerPe09}, the argument is taken from the just quoted paper.

\begin{lemma}\label{l FerPer MathScan 2009 L22}
Let $E$ be a JB$^*$-triple, $\theta >0$, $1>\delta>0$, and let $\varphi_1, \varphi_2$
be two norm-one functionals in $E^*$. Suppose $x$ is  an element in the closed unit ball of $E$,
satisfying $| {\varphi}_1 (x) | \geq 1-\delta$ and $\|x\|_{\varphi_2} \leq\theta$. Then, for every $\varepsilon>0$ with $1-\delta\geq 2 \varepsilon$ there exist two elements $\widetilde{a},y$ in the unit
ball of $E_x$, and two tripotents $u, v$ in $\left(E_{x}\right)^{**}$ such that
$\widetilde{a}\leq u \leq y \leq v=r(y)$, $1\geq  |{\varphi_1} (\widetilde{a})| > 1-\delta -\varepsilon,$
and $\|v\|_{\varphi_2} < \frac{3 \theta}{\varepsilon}$.
We can further find $a\in E_2^{**}(u)$ such that $1\geq  {\varphi_1} (a) > 1-\delta -\varepsilon.$
\end{lemma}

\begin{proof}
Let $\alpha>0$ and define $f_{\alpha}, g_{\alpha}\in C_{0} (L)$ by
$$f_{\alpha}=\left\{
\begin{array}{ll}
    0, & \hbox{if $0\leq t\leq \alpha$} \\
    \mbox{affine}, & \hbox{if $\alpha \leq t\leq 2\alpha$} \\
    t, & \hbox{if $2 \alpha\leq t\leq \|x\|$}, \\
\end{array}
\right. \   g_{\alpha}=\left\{
\begin{array}{ll}
    0, & \hbox{if $0\leq t\leq \frac{\alpha}{2}$} \\
    \mbox{affine}, & \hbox{if $\frac{\alpha}{2}\leq t\leq \alpha$} \\
    1, & \hbox{if ${\alpha}\leq t\leq \|x\|$.} \\
\end{array}
\right.$$
Let $2\eps/3<\eps'<\eps$ and define $\widetilde{a}=f_{\eps'}(x)$ and $y=g_{\eps'}(x)$ by the functional calculus
recalled in Section \ref{sec preliminaries}.
Then $\norm{\widetilde{a}}\leq1$ because $2\varepsilon'\leq1-\delta\leq\norm{x}$.
Since $\|x-\widetilde{a}\|\leq \varepsilon'$ and $| {\varphi}_1 (x)| \geq 1-\delta$ it follows that
$| {\varphi}_1 (\widetilde{a}) | > 1-\delta-\varepsilon$.\smallskip

We set $u=\chi_{[\varepsilon',\|x\|]}$, and
$v= r(y)=\chi_{(\frac{\varepsilon'}{2},\|x\|]}$ (also in $\left(E_{x}\right)^{**}$) and get $\widetilde{a}\leq u \leq y \leq v$.
Let us take  $\alpha\in \mathbb{R}$ such that ${\varphi_1} (e^{i \alpha} \widetilde{a})>0$
and define ${a} = e^{i \alpha} \widetilde{a}$.
Then $a\in E_2^{**}(u)$ because $\widetilde{a}\leq u$ and we have $$1\geq  {\varphi_1} ({a}) > 1-\delta -\varepsilon.$$
Since $\| \cdot\|_{\varphi}$ is an order-preserving map on the set
of positive elements in $\left(E_{x}\right)^{**}$ (cf. \cite[Lem.\  3.3]{FerPe06}), we deduce that
$$\| v \|_{\varphi_2}
\leq \left\| \frac{2}{\varepsilon'} x \right\|_{\varphi_2} \leq \frac{2\theta}{\varepsilon'}<\frac{3\theta}{\varepsilon}.$$
\end{proof}
\begin{proposition}\label{p orthogonal sigma-finite tripotents}
Let $E$ be a JB$^*$-triple, and let $\varphi_1$ and $\varphi_2$ be two orthogonal norm-one functionals in $E^*$.
Then for every $\varepsilon>0$ there exist norm-one elements $a,b$ in $E$ satisfying $a\perp b,$ $\varphi_1 (a) > 1-\varepsilon$ and
$\varphi_2 (b) > 1-\varepsilon$.
\end{proposition}

\begin{proof} Let us fix an arbitrary $\varepsilon>0$. Take $\eta>0$ satisfying $\eta < \min \{ \frac13, \frac{\varepsilon}{2} \}$.
We can also find $0< \delta < \frac{\varepsilon \eta}{66}.$
We note that $\eta$ and $\delta$ satisfy $2 \eta < 1-\eta,$ $1-2 \eta > 1-\varepsilon,$ and
$22 \frac{3}{\eta} \delta < \varepsilon$.\smallskip

Let $e_j$ in $E^{**}$ be the support tripotent of $\varphi_j$, $j=1,2$.
Since $e_1\perp e_2$, $e_1+e_2$ and $e_1-e_2$ are the support tripotents of $\phi=\varphi_1+\varphi_2$
and $\psi=\varphi_1-\varphi_2$, respectively (see \cite[Th. 5.4]{EdRu01}). In particular, $\varphi_1(e_2)=0=\varphi_2(e_1)$.\smallskip

By the Kaplansky Density theorem for JBW$^*$-triples \cite[Cor.\ 3.3]{BarFri90} (i.e. by the strong$^*$-density of the closed unit ball of $E$
in the one of $E^{**}$), there are two nets $(z_\lambda)$ and $(\widetilde{z}_\mu)$ in the closed unit ball of $E$ converging
in the strong$^*$-topology of $E^{**}$ to $e_1$ and $e_2$, respectively.
Since $s^*(E^{**},E^*)$ is stronger than the weak$^*$-topology of $E^{**}$, we deduce that $(z_\lambda)\to e_1$
and $(\widetilde{z}_\mu)\to e_2$ in the weak$^*$-topology of $E^{**}$. In particular,
$$\varphi_1 (z_\lambda) \to \varphi_1 (e_1)=1, \ \varphi_1 (\widetilde{z}_\mu) \to \varphi_1 (e_2)=0,$$
$$\varphi_2 (z_\lambda) \to \varphi_2 (e_1)=0, \ \varphi_2 (\widetilde{z}_\mu) \to \varphi_2 (e_2)=1,$$
$$\|z_\lambda \|_{\varphi_1} \to \|e_1\|_{\varphi_1} =1, \ \|\widetilde{z}_\mu \|_{\varphi_1} \to \|e_2\|_{\varphi_1} =0,$$
$$\|z_\lambda \|_{\varphi_2} \to \|e_1\|_{\varphi_2} =0, \hbox{ and  } \|\widetilde{z}_\mu \|_{\varphi_2} \to \|e_2\|_{\varphi_2} =1.$$
Find indices $\lambda_0$ and $\mu_0$ such that
\begin{equation}\label{eq new propo 2.6}
\betr{\varphi_1(z_{\lambda_0})}>1-\eta, \quad \|z_{\lambda_0}  \|_{\varphi_2} < \delta,
 \quad \betr{\varphi_2(\widetilde{z}_{\mu_0})-1}<\frac{3}{\eta}\delta.
\end{equation}

Applying Lemma \ref{l FerPer MathScan 2009 L22} (with $\delta,\eta, z_{\lambda_0}$ for $\theta,\delta=\eps,x$)
we can find $a_0,\widetilde{a},y$ in the closed unit ball of $E_{z_{\lambda_0}}$ and two tripotents $u,v$ in $E_{z_{\lambda_0}}^{**}$ satisfying
$$\label{eq application of L 2.2}
\widetilde{a}\leq u \leq y \leq v,\quad\ 1\geq \varphi_1 (a_0) > 1- 2 \eta >1-\varepsilon,$$ $$
\| v  \|_{\varphi_2} < \frac{3}{\eta} \delta,\hbox{ and } a_0\in E_2^{**} (u).$$
Define $a=a_0/\norm{a_0}\in E^{**}_0 (u)$. Then $\varphi_1(a)>1-\eps$.
By Lemma \ref{l FerPer MathScan 2009 L21} we obtain
$$\Big|\varphi_2 \Big(\widetilde{z}_{\mu_0}-B(y,y) (\widetilde{z}_{\mu_0})\Big)\Big| < 21 \| \widetilde{z}_{\mu_0} \| \|v\|_{\varphi_2}
< 21 \frac{3}{\eta} \delta,$$ and by the third inequality in \eqref{eq new propo 2.6} we deduce that
$$\Big|\varphi_2 \Big( B(y,y) (\widetilde{z}_{\mu_0})\Big)-1\Big| <  22 \frac{3}{\eta} \delta <\varepsilon.$$\smallskip

Setting  $\widetilde b= e^{i\beta} B(y,y)(\widetilde{z}_{\mu_0})$ for a suitable $\beta\in\R$ we have $\varphi_2 (\widetilde b )>1-  \varepsilon$
and setting $b=\widetilde b/\norm{\widetilde b}$ we still have $\varphi_2 ( b ) >1-  \varepsilon$.\smallskip

By Lemma \ref{l FerPer MathScan 2009 L.1.2}, $b\in B(y,y) (E) \subseteq E^{**}_0 (u)$. Since, by construction, $a$ lies in $E_2^{**} (u)$, it follows that $a\perp b$.
\end{proof}\smallskip
\noindent
{\bf Remark.} Proposition \ref{p orthogonal sigma-finite tripotents} remains valid (with practically the same proof)
if the first sentence is replaced by
``Let $E$ be a weak$^*$-dense subtriple of a JBW$^*$-triple $W$ and let $\varphi_1, \varphi_2$ be two orthogonal norm-one functionals
in $W_*$.''\bigskip

We shall require some results in the theory of ultraproducts of Banach spaces \cite{Hein80}. To this end, we recall some basic facts and definitions.
Let $\mathcal{U}$ be an ultrafilter on a non-empty set $I$, and let $\{X_i \}_{i\in I}$ be a family of Banach spaces. Let $\ell_{\infty} (I,X_i) = \ell_{\infty} (X_i)$ denote the Banach space obtained as the $\ell_{\infty}$-sum of the family $\{X_i \}_{i\in I}$, and let $${c}_0 \left(X_i\right) := \left\{ (x_i) \in \ell_{\infty} (X_i) : \lim_{\mathcal{U}} \|x_i\| =0 \right\}.$$
The ultraproduct of the family $\{X_i \}_{i\in I}$ relative to the ultrafilter $\mathcal{U}$, denoted by $(X_i)_\mathcal{U}$,
is the quotient Banach space $\ell_{\infty} (X_i)/{c}_0 \left(X_i\right)$ equipped with the quotient norm. Let $[x_i]_{\mathcal{U}}$
be an equivalence class in $(X_i)_\mathcal{U}$ represented by a family $(x_i)_i \in \ell_{\infty} (X_i)$.
It is known that $$\| [x_i]_{\mathcal{U}} \| = \lim_{\mathcal{U}} \|x_i\|,$$ independently of the representative of $[x_i]_{\mathcal{U}}$.
In general, the ultraproduct of a family of dual Banach spaces is not a dual Banach space (not even in the case of von Neumann algebras).
The ultraproduct $(X_i^*)_\mathcal{U}$ of the
duals can be identified isometrically with a closed subspace of the dual $\displaystyle\left((X_i)_\mathcal{U}\right)^*$
via the canonical mapping $$\mathcal{J} : (X_i^*)_\mathcal{U} \to  \left((X_i)_\mathcal{U}\right)^*$$
$$\mathcal{J} [\varphi_i]_{\mathcal{U}} ([x_i]_{\mathcal{U}}) = \lim_{\mathcal{U}} \varphi_i (x_i).$$
In \cite[Cor.\ 10]{Dineen86} S. Dineen establishes that the class of JB$^*$-triples
(analogously to the class of C$^*$-algebras \cite[Prop.\ 3.1]{Hein80}), is stable under ultraproducts via the canonical triple product
$\{[u_i]_{\mathcal{U}},[v_i]_{\mathcal{U}},[w_i]_{\mathcal{U}}\}=[\{u_i,v_i,w_i\}]_{\mathcal{U}}$.\smallskip

Here is a simple argument to prove Dineen's theorem (cf.\ also \cite[proof of Cor.\ 3.3.5]{Chu2012}).
Let $\{E_i \}_{i\in I}$ be a family of JB$^*$-triples.
Then the Banach space $\ell_{\infty} (E_i)$ is a JB$^*$-triple with pointwise operations (\cite[page 523]{Kaup83} or \cite[Ex.\ 3.1.4]{Chu2012}).
Let $E$ be a JB$^*$-triple.
A subtriple $\mathcal{I}$ of a JB$^*$-triple a JB$^*$-triple $E$ is said to be an \emph{ideal} or a \emph{triple ideal} of $E$
if $\{ E,E,\mathcal{I}\} + \{ E,\mathcal{I},E\}\subseteq \mathcal{I}$.
It is easy to see, under the above conditions, that
$\displaystyle\left\{\ell_{\infty} (E_i), \ell_{\infty} (E_i), {c}_0 \left(E_i\right)\right\} \subseteq {c}_0 \left(E_i\right)$
and $\displaystyle\left\{\ell_{\infty} (E_i), {c}_0 \left(E_i\right),  \ell_{\infty} (E_i)\right\} \subseteq {c}_0 \left(E_i\right),$
and hence ${c}_0 \left(E_i\right)$ is a closed triple ideal of $\ell_{\infty} (E_i)$.
Since the quotient of a JB$^*$-triple by a closed triple ideal is a JB$^*$-triple (\cite{Kaup83} or \cite[Cor.\ 3.1.18]{Chu2012}),
we deduce that $(E_i)_{\mathcal{U}}= \ell_{\infty} (E_i)/{c}_0 \left(E_i\right)$ is a JB$^*$-triple.\smallskip

\begin{proposition}\label{p quantitative version of L-orthogonality in the dual space of a JB*-triple} For each $\varepsilon>0$ there exists
$\delta > 0$ such that for every JB$^*$-triple $E$
and every pair of functionals $\varphi_1$ and $\varphi_2$ in the closed unit ball of $E^*$ with
$2\geq \|\varphi_1  \pm \varphi_2 \| \geq 2 (1-\delta)$
there exist orthogonal norm-one elements $a,b$ in
$E$ satisfying $\varphi_1 (a) > 1-\varepsilon$ and $\varphi_2 (b) > 1-\varepsilon$.
\end{proposition}

\begin{proof}
Suppose, to the contrary, that there exists $\varepsilon_0>0$ such that for each natural $n$, we can find a JB$^*$-triple $E_n$ and functionals
$\varphi_{1,n}$ and $\varphi_{2,n}$ in the closed unit ball of $E_n^*$ with $2\geq \|\varphi_{1,n} \pm \varphi_{2,n}\| \geq 2 (1-\frac1n )$
satisfying $|\varphi_{1,n} (a)| \leq  1-\varepsilon_0$ and $|\varphi_{2,n} (b)| \leq  1-\varepsilon_0$, whenever $a,b$ are elements of norm one
in $E_n$ with $a\perp b$.\smallskip

Take a non-trivial ultrafilter $\mathcal{U}$ in $\mathbb{N}$, let $0<\eps_1<\eps$ and
let $\mathcal{J} : (E_n^*)_\mathcal{U} \to  \left((E_n)_\mathcal{U}\right)^*$
be the canonical isometric embedding defined by
$\mathcal{J} [\varphi_i]_{\mathcal{U}} ([x_n]_{\mathcal{U}}) = \lim_{\mathcal{U}} \varphi_n (x_n)$.
Then $\mathcal{J}[\varphi_{1,n}]_{\mathcal{U}}$ and $\mathcal{J}[\varphi_{2,n}]_{\mathcal{U}}$ have norm one and are $L$-ortho-gonal\hyphenation{ortho-gonal}
in $\displaystyle\left((E_n)_\mathcal{U}\right)^*$ because so are $[\varphi_{1,n}]_{\mathcal{U}}$ and $[\varphi_{2,n}]_{\mathcal{U}}$
in $(E^*_n)_{\mathcal{U}}$.
As explained above, $\left((E_n)_\mathcal{U}\right)^*$ is a JB$^*$-triple and Proposition \ref{p orthogonal sigma-finite tripotents} applies:
there exist norm-one elements $[a_n]_{\mathcal{U}},[b_n]_{\mathcal{U}}$ in $(E_n)_\mathcal{U}$ satisfying
$[a_n]_{\mathcal{U}}\perp [b_n]_{\mathcal{U}},$
$\mathcal{J}[\varphi_{1,n}]_{\mathcal{U}} ([a_n]_{\mathcal{U}}) > 1-\varepsilon_1$ and
$\mathcal{J}[\varphi_{2,n}]_{\mathcal{U}} ([b_n]_{\mathcal{U}}) > 1-\varepsilon_1$.\smallskip

We note that the elements $[a_n]_{\mathcal{U}},[b_n]_{\mathcal{U}}$ are orthogonal in the quotient
$(E_n)_\mathcal{U}  = \ell_{\infty} (E_n)/{c}_0 \left(E_n\right)$. Since the quotient mapping
$\pi : \ell_{\infty} (E_n)\to \ell_{\infty} (E_n)/{c}_0 \left(E_n\right)$ is a triple homomorphism between JB$^*$-triples and
$\pi ((a_n)_n)$ $= [a_n]_{\mathcal{U}}$ $\perp \pi ((b_n)_n)= [b_n]_{\mathcal{U}}$, by \cite[Proposition 4.7]{BunPe07} there exist orthogonal elements
$(\widetilde{a}_n)_n$ and $(\widetilde{b}_n)_n$ in $\ell_{\infty} (E_n)$ satisfying $\pi ((\widetilde{a}_n)_n)= [a_n]_{\mathcal{U}}$ and
$\pi ((\widetilde{b}_n)_n)= [b_n]_{\mathcal{U}}$. We have $\lim_{\mathcal{U}}\norm{\widetilde a_n}=\lim_{\mathcal{U}}\norm{a_n}=1$
and likewise for $(b_n)_n$, $(\widetilde b_n)_n$.\smallskip

Now, $1-\varepsilon_1 < \mathcal{J}[\varphi_{1,n}]_{\mathcal{U}} ([\widetilde{a}_n]_{\mathcal{U}}) = \lim_{\mathcal{U}}  \varphi_{1,n} (\widetilde{a}_n)$,
$ 1-\varepsilon_1< \mathcal{J}[\varphi_{2,n}]_{\mathcal{U}} ([\widetilde{b}_n]_{\mathcal{U}}) =  \lim_{\mathcal{U}}  \varphi_{2,n} (\widetilde{b}_n)$,
and, for every $n$, $\widetilde{a}_n \perp \widetilde{b}_n$. Hence $\varphi_{1,n}(\widetilde{a}_n/\norm{\widetilde{a}_n})>1-\eps$ or
$\varphi_{2,n}(\widetilde{b}_n/\norm{\widetilde{b}_n})>1-\eps$ can be achieved for infinitely many $n$
which contradicts the assumption made in the beginning of the proof.
\end{proof}

\noindent
We shall establish now an analogous version of Proposition \ref{p quantitative version of L-orthogonality in the dual space of a JB*-triple}
for finite sets of functionals in the dual of a JB$^*$-triple.

\begin{theorem}\label{t quantitative version of L-orthogonality in the dual space of a JB*-triple n-functionals}
For each $\varepsilon>0$ and each natural $n$, there exists $\delta = \delta(n,\varepsilon) > 0$ with the following property.
Let $E$ be a JB$^*$-triple and let $\varphi_1,\ldots, \varphi_n$ be functionals in $E^*$ such that
\bgl
\sum_{j=1}^{n} |\alpha_j| \geq \left\|\sum_{j=1}^{n} \alpha_j \varphi_j  \right\|
\geq \left(1-\delta(n,\varepsilon)\right) \sum_{j=1}^{n} |\alpha_j| \quad\forall \alpha_j\in \mathbb{C}.\label{eq gl2.9.4}
\egl
Then there exist mutually orthogonal elements $a_1,\ldots, a_n$ of norm one in $E$ and
mutually orthogonal functionals $\widetilde\varphi_1,\ldots,\widetilde\varphi_n$ of norm one in $E^*$ satisfying
\bgl\label{eq gl2.9.1b}
\varphi_j (a_j) > 1-\varepsilon\quad\mbox{and}\quad\| \varphi_j- \widetilde\varphi_j\| <\varepsilon\quad\forall j=1,\ldots, n\\
\mbox{where}\quad\widetilde\varphi_j=\frac{{\varphi}_j P_2 (r(a_j))}{\|{\varphi}_j P_2 (r(a_j))\|}.\nonumber
\egl
\end{theorem}

\begin{proof}
We shall proceed by induction over $n\geq 1$.
For $n=1$ there is nothing to prove.
Let us fix $\varepsilon>0$ and $n\in\N$.\smallskip

We {\sc claim} that there is $\eps'\in(0,\eps)$ such that if an element $b$ in a JB$^*$-triple $E$ and $\xi\in E^*$ satisfy
\bgl \norm{b}\leq1,\quad\norm{\xi}\leq1\quad\mbox{ and }\quad\betr{\xi(b)}>1-\eps'\label{eq gl2.9.2}\egl then $\Norm{\xi-\frac{\xi P_2(r(b))}{\norm{\xi P_2(r(b))}}}<\eps$
In fact, define $\eps'\in(0,\eps/2)$ by Lemma \ref{l BeRo Lemma 2.2} such that \eqref{eq gl2.9.2} entails $\norm{\xi-\psi}<\eps/2$
where $\psi=\xi P_2(r(b))$. Hence, if \eqref{eq gl2.9.2} holds then $\betr{\psi(b)}=\betr{\xi((b)}>1-\eps/2$ and
\bglst
\Norm{\xi-\frac{\psi}{\norm{\psi}}}
&\leq&\norm{\xi-\psi}+\Norm{\psi-\frac{\psi}{\norm{\psi}}}<\frac{\eps}{2}+\left(\frac{1}{\norm{\psi}}-1\right)\norm{\psi}\\
&=&\frac{\eps}{2}+1-\norm{\psi}<\eps
\eglst
which proves the claim.\smallskip

Choose $\delta(n,{\varepsilon'}) > 0$ according to the induction hypothesis, choose
$\eta_0=\eta(\frac{\delta(n,{\varepsilon'})}{2})> 0$ according to Lemma \ref{l BeRo Lemma 2.2}
and choose
$$\delta_0=\delta \left(\min\left\{{\varepsilon'}, \frac1n \eta_0\right\}\right)>0$$
according to Theorem \ref{p quantitative version of L-orthogonality in the dual space of a JB*-triple}.
Furthermore, let $\delta(n+1,\varepsilon')$ be such that
$$0<\delta(n+1,\varepsilon') <\min \left\{\frac{\delta(n,{\varepsilon'})}{2} , \delta_0 \right\}.$$
Let $E$ be a JB$^*$-triple and let $\varphi_1,\ldots, \varphi_{n+1}\in E^*$ satisfy
$$\sum_{j=1}^{n+1} |\alpha_j| \geq \left\|\sum_{j=1}^{n+1} \alpha_j \varphi_j  \right\|
\geq \Big(1-\delta(n+1,\varepsilon')\Big) \sum_{j=1}^{n+1} |\alpha_j| \quad\forall\alpha_j\in \mathbb{C}.$$
Let us define $\displaystyle \phi = \sum_{j=1}^{n} \frac1n \varphi_j$. Clearly,
$$2\geq \|\phi  \pm \varphi_{n+1} \|\geq \Big(1-\delta(n+1,\varepsilon')\Big) (n \frac1n + 1)> 2(1-\delta_0).$$
Thus, by the choice of $\delta_0$ (see Theorem \ref{p quantitative version of L-orthogonality in the dual space of a JB*-triple})
there exist $a\perp a_{n+1}$ of norm one in $E$ satisfying
\bgl
\varphi_{n+1} (a_{n+1}) &>& 1- \min\left\{{\varepsilon'}, \frac1n \eta_0\right\}>1-\varepsilon',\label{eq gl2.9.3}\\
\phi (a) &>& 1- \min\left\{{\varepsilon'}, \frac1n \eta_0\right\}>1-\frac1n \eta_0.\nonumber
\egl
Since $\sum_{j=1}^{n} \varphi_j (a)=n\phi (a) >n- \eta_0$ and
$\|\varphi_j\|\leq 1$ we have
$$\left|\varphi_j (a)  \right| > n - \eta_0  - (n-1)
= 1- \eta_0\quad\forall j=1,\ldots,n.$$
Thus, $\left|\varphi_j  P_2(r(a)) (a)  \right|=  \left|\varphi_j (a)  \right| >  1- \eta_0$ and by the choice of $\eta_0$ we deduce
$$\left\| \varphi_j  - \varphi_j P_2 (r(a))\right\| < \frac{\delta(n,{\varepsilon'})}{2}\quad\forall j=1,\ldots, n.$$
Therefore we have
$$\sum_{j=1}^{n} |\alpha_j| \geq \left\|\sum_{j=1}^{n} \alpha_j \varphi_j P_2 (r(a))  \right\|$$
$$ \geq \left\|\sum_{j=1}^{n} \alpha_j \varphi_j  \right\| - \left\|\sum_{j=1}^{n} \alpha_j \Big(\varphi_j  -\varphi_j P_2 (r(a))\Big)  \right\|$$
$$\geq \Big(1-\delta(n+1,\varepsilon')\Big) \sum_{j=1}^{n} |\alpha_j| - \frac{\delta(n,{\varepsilon'})}{2} \sum_{j=1}^{n} |\alpha_j|
\geq \Big(1-\delta(n,{\varepsilon'})\Big) \sum_{j=1}^{n} |\alpha_j|$$
for all scalars $\alpha_j$.
Recall that $r(a)$ is an open tripotent which means that the subtriple $F:=E\cap E^{**}_2 (r(a))$ is weak$^*$-dense in $E^{**}_2 (r(a))$.
Set $\psi_j=\varphi_j P_2 (r(a))_{|F}$ for $j\leq n$. Then
\bglst
\sum_{j=1}^{n} |\alpha_j| \geq \left\|\sum_{j=1}^{n} \alpha_j \psi_j\right\|=\left\|\sum_{j=1}^{n} \alpha_j\varphi_j P_2(r(a))\right\|
\geq\Big(1-\delta(n,{\varepsilon'})\Big) \sum_{j=1}^{n} |\alpha_j|,
\eglst
for all $\alpha_j\in\mathbb{C},$ and by the induction hypothesis, applied to $F$, there exist mutually orthogonal norm-one elements $a_1,\ldots, a_n\in F$
satisfying $\psi_j(a_j)=\varphi_j (a_j) > 1-\varepsilon'$, for every $j=1,\ldots, n$. They are orthogonal to $a_{n+1}$ because $a$ is.
Together with \eqref{eq gl2.9.3} this shows the first half of \eqref{eq gl2.9.1b} (for $n+1$) because $1-\eps'>1-\eps$.
The second half follows from the claim. This ends the induction and the proof.
\end{proof}
\noindent
In passing we note an obvious reformulation of the conclusion of Theorem \ref{t quantitative version of L-orthogonality in the dual space of a JB*-triple n-functionals}:
There exists an abelian subtriple $\mathcal{C}$ of $E$ such that if we set $\psi_j={\varphi_j}_{|\mathcal{C}}$ then $(\psi_j)_{j=1}^n$
spans $\ell_1(n)$ $(1-\eps)$-isomorphically in $\mathcal{C}^*$.\bigskip\\
For the proofs of Theorems \ref{t asymptotically isometric copies of ell_1 JBW*-triple preduals} and \ref{t isomorphic copies in the dual of a JB-triple}
we need the following technical strengthening of Theorem \ref{t quantitative version of L-orthogonality in the dual space of a JB*-triple n-functionals}.

\begin{lemma}\label{l compact tripotents in Cor 2.11}
In Theorem \ref{t quantitative version of L-orthogonality in the dual space of a JB*-triple n-functionals}
the $a_j$ can be constructed such that additionally
there are mutually orthogonal compact tripotents $u_1,\ldots, u_n$ in $E^{**}$
such that $a_j\in E^{**}_2 ( u_j)$ for $j=1,\ldots,n$.
\end{lemma}
\begin{proof}
We {\sc claim} that if a JB$^*$-triple $E$, $a\in E$, $\varphi\in E^*$, and $\eps'>0$ are given such that
\bgl
\norm{\varphi}\leq1,\quad\norm{a}\leq1,\quad\varphi(a)>1-\eps'\label{eq gl2.11.1}
\egl
then there exist a compact tripotent $u\in E_a^{**}\subset E_2^{**}(e)$ and $b\in E_a\cap E_2^{**}(u)$ such that $\norm{b}=1$ and
$\varphi(b)>1-\eps'$.\smallskip

In order to show the claim suppose \eqref{eq gl2.11.1} holds.
Define $z_m=f_{\alpha}(a)\in E_a$ for $\alpha=\norm{b_j}/m$ where $f_{\alpha}$ is as in the proof of Lemma \ref{l FerPer MathScan 2009 L22}.
Since $\norm{z_m-a}\to0$ there is $m_0$ such that $\varphi(e^{i\theta}z_{m_0})>1-\eps'$ for an appropriate $\theta\in\R$.
Also $\norm{z_{m_0}}\leq1$. It remains to set $b=e^{i\theta}z_{m_0}/\norm{z_{m_0}}$ and $u= \chi_{_{[{\|b_j\|}/{2 m_0}, \|b_j\|]\cap L}}$
and the claim is proved.\smallskip

Now we apply the claim $n$ times to pairwise orthogonal $a_j$ and note that $u_j\in E_{a_j}^{**}\perp E_{a_k}^{**}\ni u_k$ if $j\neq k$.
The claim in the proof of Theorem \ref{t quantitative version of L-orthogonality in the dual space of a JB*-triple n-functionals}
shows that it is enough to replace $a_j$ by $b_j$ in \eqref{eq gl2.9.1b} in order to finish the proof.
\end{proof}\noindent
Recalling that Peirce projections associated with a tripotent in a JBW$^*$-triple are weak$^*$-continuous,
and the fact that the range tripotent of an element in a JBW$^*$-triple always lies in the JBW$^*$-triple, the arguments given above show:

\begin{corollary}\label{c perturbation of L-orthogonality in the predual space of a JBW*-triple}
For each $\varepsilon>0$ and each natural $n$, there exists a positive $\delta = \delta(n,\varepsilon)$ such that
for every JBW$^*$-triple $W,$ and every finite set of functionals $\varphi_1,\ldots, \varphi_n$ in $W_*$ satisfying \eqref{eq gl2.9.4}
there exist orthogonal norm one elements $a_1,\ldots, a_n\in W$  and orthogonal functionals $\widetilde{\varphi}_1,\ldots,\widetilde{\varphi}_n\in W_*$
of norm one such that \eqref{eq gl2.9.1b} holds.
\end{corollary}\noindent
The following corollary will not be needed in the sequel but could perhaps be useful elsewhere. The argument leading to part (a)
has already been used in the proof of Theorem \ref{t quantitative version of L-orthogonality in the dual space of a JB*-triple n-functionals}.
\begin{corollary}\label{r Theorems with open tripotent}
For each $\varepsilon>0$ and each natural $n$, there exists a positive $\delta = \delta(n,\varepsilon)$ with the following properties.
\begin{enumerate} \item[(a)]
Let $E$ be a JB$^*$-triple, $e\in E^{**}$ an open tripotent and let $\varphi_1,\ldots, \varphi_n$ be functionals in the closed unit ball of $E^*$.
If
\bgl\label{eq gl2.12}
\sum_{j=1}^{n} |\alpha_j| \geq \left\|\sum_{j=1}^{n} \alpha_j \varphi_j P_2 (e)  \right\|
\geq \left(1-\delta(n,\varepsilon)\right) \sum_{j=1}^{n} |\alpha_j|, \quad\forall \alpha_j\in \mathbb{C}
\egl
then there exist orthogonal norm-one elements $a_1,\ldots, a_n$ in $E\cap E^{**}_2 (e)$ and
mutually orthogonal norm-one functionals $\widetilde{\varphi}_1,$ $\ldots,$ $\widetilde{\varphi}_n$ in $E^*$ satisfying \eqref{eq gl2.9.1b}.
\item[(b)]
Let $W$ be a JBW$^*$-triple, $e\in E^{**}$ an arbitrary tripotent and let $\varphi_1,\ldots, \varphi_n$ be functionals in the closed unit ball of
$W_*$ satisfying \eqref{eq gl2.12}.
Then there exist mutually orthogonal elements $a_1,\ldots, a_n$ of norm one in $W_2 (e)$
and mutually orthogonal norm-one functionals $\widetilde{\varphi}_1,\ldots,\widetilde{\varphi}_n$ in $W_*$ satisfying \eqref{eq gl2.9.1b}.
\end{enumerate}
\end{corollary}
\begin{proof}(a)
Set $F= E\cap E^{**}_2 (e)$ and $\psi_j={\varphi_j P_2(e)}_{|F}$. Then $\overline{F}^{w^*}= E^{**}_2 (e)$ because $e$ is open.
The $\psi_j$ satisfy \eqref{eq gl2.9.4} and it is enough to apply Theorem \ref{t quantitative version of L-orthogonality in the dual space of a JB*-triple n-functionals}
to $F$. Similarly, for part (b) identify $\varphi_j P_2 (e)$ with $\varphi_{|W_2(e)}\in (W_2(e))_*$ and apply
Corollary  \ref{c perturbation of L-orthogonality in the predual space of a JBW*-triple} to $W_2(e)$.
\end{proof}

\noindent
We can now establish the promised quantitative version of the last statement in Lemma \ref{l EdRu01 L-orthogonal}.

\begin{proposition}\label{p quantitative version of M-orthogonal tripotents} Given $\varepsilon>0$ there exits $\delta = \delta(\varepsilon)$ satisfying that for every JB$^*$-triple $E$ and every couple of tripotents $u,v$ in $E$ with $$1-\delta < \|u\pm v\|< 1+\delta,$$ we have $\|u -P_0 (v)(u)\|< \varepsilon$ and $\|v -P_0 (u)(v)\|< \varepsilon$.
\end{proposition}

\begin{proof} Let us note that the statement is true whenever the set of tripotents in a JB$^*$-triple $E$ reduces to the zero element. Suppose, contrary to our claim, that there exists $\varepsilon_0>0$ such that for each natural $n$, we can find a JB$^*$-triple $E_n$ and tripotents
$u_{n}$ and $v_{n}$ with $1-\frac1n < \|u_{n} \pm v_{n}\| < 1+\frac1n $
satisfying $\|u_n -P_0 (v_n)(u_n)\|\geq \varepsilon_0$ or $\|v_n -P_0 (u_n)(v_n)\|\geq \varepsilon_0.$\smallskip

Take a non-trivial ultrafilter $\mathcal{U}$ in $\mathbb{N}$. The elements $[u_{n}]_{\mathcal{U}}$ and $[v_{n}]_{\mathcal{U}}$
are non-zero tripotents in $(E_n)_{\mathcal{U}}$ with $\left\|[u_{n}]_{\mathcal{U}} \pm [v_{n}]_{\mathcal{U}} \right\| = 1$, that is,
$[u_{n}]_{\mathcal{U}}\perp_{M} [v_{n}]_{\mathcal{U}}$ in $(E_n)_{\mathcal{U}}$. The final statement in Lemma \ref{l EdRu01 L-orthogonal}
implies that $[u_{n}]_{\mathcal{U}}\perp [v_{n}]_{\mathcal{U}}$ in $(E_n)_{\mathcal{U}}$.
In particular $$ [P_0(u_{n}) (v_{n}) ]_{\mathcal{U}} = P_0 ([u_{n}]_{\mathcal{U}}) ( [v_{n}]_{\mathcal{U}}) =[v_{n}]_{\mathcal{U}}$$
(because $P_0 ([u_{n}]_{\mathcal{U}})=[P_0(u_{n}) ]_{\mathcal{U}}$) and
$$[P_0(v_{n}) (u_{n}) ]_{\mathcal{U}} = P_0 ([v_{n}]_{\mathcal{U}}) ( [u_{n}]_{\mathcal{U}}) =[u_{n}]_{\mathcal{U}},$$
which implies that $\displaystyle\lim_{\mathcal{U}} \|P_0(u_{n}) (v_{n}) -v_n\| =0 $ and
$\displaystyle\lim_{\mathcal{U}} \|P_0(v_{n}) (u_{n}) -u_n\| =0,$ contradicting our assumptions.
\end{proof}

\section{Infinite dimensional copies of $\ell_1$ in preduals of JBW$^*$-triples}\label{sec 4}
\begin{theorem}\label{t asymptotically isometric copies of ell_1 JBW*-triple preduals}
Let $W$ be a JBW$^*$-triple and let $(\varphi_m)$ be a bounded sequence in its predual $W_*$. If $(\varphi_m)$ spans $\ell_1$
almost isometrically then there are a subsequence $(\varphi_{m_n})$ of $(\varphi_m)$ and a sequence $(\widetilde{\varphi}_n)$
of pairwise orthogonal functionals in $W_*$ such that $\|\varphi_{m_n} -  \widetilde{\varphi}_n\| \to 0$ when $l\to \infty$.
\end{theorem}

\begin{proof}
We can assume that $\|\varphi_m\| =1$, for every $m$. Let $(\nu_n)$ be a sequence of strictly positive numbers
such that $\sum_{n=1}^{\infty} \nu_n <\infty$.
We shall prove, by induction over $n$, the existence of $m_n\in \mathbb{N},$ and $\phi_{m_1}^{(n)},\ldots, \phi_{m_n}^{(n)}$ in $W_*$
satisfying $m_{n-1} < m_n$, and for each natural $n$
$$  \phi_{m_k}^{(n)}\perp \phi_{m_l}^{(n)}\quad\forall k\neq l\in\{1,\ldots, n\},$$
$$\|\phi_{m_k}^{(n)}\| =1,\quad\forall k\leq n,$$
$$\|\phi_{m_k}^{(n)}-\phi_{m_k}^{(n-1)}\| < \nu_n,\quad \forall k=1,\ldots, n-1\quad \mbox{if }n\geq2,$$ and
$\|\phi_{m_n}^{(n)}- \varphi_{m_n}\| < \nu_n$ .\smallskip

When $n=1$ we set $m_1=1$, $\phi_{1}^{(1)}= \varphi_1$ and the statement is clear. Suppose that $m_1<m_2<\ldots<m_n$, $\{\phi_{m_1}^{(1)}\}$, $\{\phi_{m_1}^{(2)}, \phi_{m_2}^{(2)}\}$, ..., $\{\phi_{m_1}^{(n)},\ldots, \phi_{m_n}^{(n)}\}$ have been defined satisfying the above properties.\smallskip

By Corollary \ref{c perturbation of L-orthogonality in the predual space of a JBW*-triple}, there exists $\delta_1 = \min\{\delta(n, \nu_{n+1}/2),\nu_{n+1}/2\}>0$.
Choose a natural $j$ satisfying  $\frac{21}{\sqrt{j}}  < \delta_1$.
We use Corollary \ref{c perturbation of L-orthogonality in the predual space of a JBW*-triple} again in order to choose
$\delta_0 = \delta(n j , \nu_{n+1})>0$.
Since $(\varphi_m)$ spans $\ell_1$ almost isometrically there exists $m_0 > m_n$ satisfying
$$ (1-\delta_0) \sum_{m=m_0}^{\infty} |\alpha_m| \leq \left\|\sum_{m=m_0}^{\infty} \alpha_m \varphi_m \right\|\quad\forall\alpha_m\in\mathbb{C}.$$
Set $N= \{m_0+1, \ldots, m_0 + n j \}\subseteq \mathbb{N}.$ Since
$$ (1-\delta(n j , \nu_{n+1})) \sum_{m=m_0+1}^{m_0 + n j } |\alpha_m| \leq \left\|\sum_{m=m_0+1}^{m_0 + n j } \alpha_m \varphi_m \right\|\quad\forall\alpha_m\in\mathbb{C},$$
Corollary \ref{c perturbation of L-orthogonality in the predual space of a JBW*-triple} implies the existence of mutually orthogonal elements
$a_1,\ldots, a_{n j}$ in the closed unit ball of $W$ such that
\begin{equation}\label{eq one 2601} \Big\| \varphi_m- \frac{{\varphi}_m P_2 (r(a_m))}{\|{\varphi}_m P_2 (r(a_m))\|}\Big\| <\nu_{n+1}
\quad\forall m\in N.
\end{equation}

On the other hand, it is clear, by orthogonality, that
$$0\leq \sum_{m\in N} \sum_{k=1}^{n}   \left\|r(a_m) \right\|_{\phi_{m_k}^{(n)}}^2 = \sum_{k=1}^{n}  \sum_{m\in N} \left\|r(a_m) \right\|_{\phi_{m_k}^{(n)}}^2 =  \sum_{k=1}^{n} \left\| \sum_{m\in N} r(a_m) \right\|_{\phi_{m_k}^{(n)}}^2 \leq n.$$
Thus, there exists $m_{n+1}\in N$ satisfying
$$ \left\|r(a_{m_{n+1}}) \right\|_{\phi_{m_k}^{(n)}}^2 \leq \frac1j \quad\forall k=1,\ldots, n$$
hence, by Lemma \ref{l FerPer MathScan 2009 L21},
\begin{equation}\label{eq two 2501} \| \phi_{m_k}^{(n)} - \phi_{m_k}^{(n)}P_0 (r(a_{m_{n+1}})) \| \leq 21 \frac{1}{\sqrt{j}}
\quad\forall k=1,\ldots, n.
\end{equation}
We define $\widetilde{\phi}_{m_k}^{(n+1)} = \phi_{m_k}^{(n)} P_0 (r(a_{m_{n+1}})),$ for $k=1,\ldots, n$
and $${\phi}_{m_{n+1}}^{(n+1)} = \frac{\varphi_{m_{n+1}} P_2 (r(a_{m_{n+1}}))}{\| \varphi_{m_{n+1}} P_2 (r(a_{m_{n+1}})) \|}.$$
By \eqref{eq one 2601}, $\Big\| \varphi_{m_{n+1}}- {\phi}_{m_{n+1}}^{(n+1)} \Big\| <\nu_{n+1},$ and by \eqref{eq two 2501}
\begin{equation}\label{eq 3 2601} \| \phi_{m_k}^{(n)} - \widetilde{\phi}_{m_k}^{(n+1)} \| \leq 21 \frac{1}{\sqrt{j}}< \delta_1
\leq \frac{\nu_{n+1}}{2} \quad\forall k=1,\ldots, n.
\end{equation}
Therefore we have
\bglst
\sum_{k=1}^{n} |\alpha_k|
&\geq& \left\| \sum_{k=1}^n \alpha_k \widetilde{\phi}_{m_k}^{(n+1)} \right\| 
\geq \left\| \sum_{k=1}^n \alpha_k {\phi}_{m_k}^{(n)} \right\|- \left\| \sum_{k=1}^n \alpha_k (\widetilde{\phi}_{m_k}^{(n+1)}-{\phi}_{m_k}^{(n)}) \right\|\\
&\geq&\left\| \sum_{k=1}^n \alpha_k {\phi}_{m_k}^{(n)} \right\| -\delta_1 \sum_{k=1}^{n} |\alpha_k| = \sum_{k=1}^{n} |\alpha_k|-\delta_1 \sum_{k=1}^{n} |\alpha_k| \\
&\geq& (1-\delta(n, \nu_{n+1}/2)) \sum_{k=1}^{n} |\alpha_k|    \quad\quad\forall\alpha_m\in\mathbb{C}.
\eglst
By Corollary \ref{c perturbation of L-orthogonality in the predual space of a JBW*-triple}, applied to the JBW$^*$-triple $W_0 (r(a_{m_{n+1}}))$
and the functionals $\{\widetilde{\phi}_{m_1}^{(n+1)},\ldots, \widetilde{\phi}_{m_n}^{(n+1)}\}\subset \left(W_0 (r(a_{m_{n+1}})) \right)_*$,
we can find mutually orthogonal elements $b_1,\ldots, b_{n}$ in the closed unit ball of $W_0 (r(a_{m_{n+1}}))$ such that
\begin{equation}\label{eq 4 2601}
\Big\| \widetilde{\phi}_{m_k}^{(n+1)} - \frac{\widetilde{\phi}_{m_k}^{(n+1)}P_2 (r(b_j))}{\|\widetilde{\phi}_{m_k}^{(n+1)} P_2 (r(b_j))\|}\Big\|
<\frac{\nu_{n+1}}{2} \quad\forall k=1,\ldots,n.
\end{equation}
We define ${\phi}_{m_k}^{(n+1)} := \frac{\widetilde{\phi}_{m_k}^{(n+1)}P_2 (r(b_j))}{\|\widetilde{\phi}_{m_k}^{(n+1)} P_2 (r(b_j))\|}$
for $k=1,\ldots,n$.
The inequalities \eqref{eq 3 2601} and \eqref{eq 4 2601} show $\|\phi_{m_k}^{(n)}-\phi_{m_k}^{(n+1)}\| < \nu_n$ ($k=1,\ldots,n$), which finishes the induction argument.\smallskip

Fix a natural $k$ and consider the sequence $(\phi_{m_k}^{(n)})_{n\geq k}$.
The inequalities $\|\phi_{m_k}^{(n)}-\phi_{m_k}^{(n-1)}\| < \nu_n$ and
$$\|\phi_{m_k}^{(n)}-\phi_{m_k}^{(i)}\| <   \sum_{j=i+1}^{n} \nu_j \to 0 \mbox{ if }n>i\to\infty$$
show that $(\phi_{m_k}^{(n)})_{n\geq k}$ is a Cauchy sequence which converges to some $\widetilde{\varphi}_k\in W_*$.
By construction $\phi_{m_k}^{(n)} \perp \phi_{m_j}^{(n)}$ for every $k\neq j$, and every $n\geq \max\{j,k\}$, therefore
$$\|\widetilde{\varphi}_k \pm \widetilde{\varphi}_j \| = \lim_{n\to\infty} \|\phi_{m_k}^{(n)} \pm \phi_{m_j}^{(n)}\| = 2$$
for every $k\neq j$ in $\mathbb{N}$. This implies $\widetilde{\varphi}_k \perp \widetilde{\varphi}_j$ for every $j\neq k$
(cf. Lemma \ref{l EdRu01 L-orthogonal}).
Finally, the inequality
$$\|\varphi_{m_n} - \widetilde{\varphi}_n \| \leq \|\phi_{m_n}^{(n)}- \varphi_{m_n}\| + \|\phi_{m_n}^{(n)} -  \widetilde{\varphi}_n\|$$
$$= \|\phi_{m_n}^{(n)}- \varphi_{m_n}\| + \|\phi_{m_n}^{(n)} -  \lim_{k\to\infty} \phi_{m_n}^{(k)} \| < \nu_n+ \sum_{k=n+1}^{\infty} \nu_{k},$$
gives the desired statement $\displaystyle \lim_{n\to \infty} \|\varphi_{m_n} - \widetilde{\varphi}_n \| =0$.
\end{proof}

The study of isomorphic copies in the dual space of a JB$^*$-triple requires an extra effort. It should be remarked here that
the next proposition can be considered as a quantitative version of \cite[Theorem 1]{Pfi94} and \cite[Theorem 2.3]{FerPe09}.

\begin{theorem}\label{t isomorphic copies in the dual of a JB-triple} Let $E$ be a JB$^*$-triple and let $(\varphi_m)$ be a
normalized sequence in $E^*$ spanning $\ell_1$ $r$-isomorphically {\rm(}with $0<r\leq 1)$. Then for each $\varepsilon>0$
there exist a subsequence $(\varphi_{m_n})$ of $(\varphi_m)$ and a sequence $(c_n)$ of mutually orthogonal elements of norm one in $E$
such that
\bgl
\varphi_{m_n} (c_n)> r (1-\varepsilon),\quad\quad\forall n\in\mathbb{N},\label{eq gl4.2}
\egl
and such that the restrictions ${\varphi_{m_n}}_{|\mathcal{C}}$ span $\ell_1$ $(r(1-\eps))$-isomorphically where $\mathcal{C}$ is the
abelian subtriple of $E$ generated by the $c_n$'s and isometric to a commutative C$^*$-algebra.
\end{theorem}

\begin{proof} We may assume that $1\geq \varepsilon>0$, we consider a series $\displaystyle \sum_{n\geq 1} \varepsilon_n$ with
$\varepsilon_n>0$ and $\displaystyle \sum_{n= 1}^{\infty} \varepsilon_n = \frac{\varepsilon}{2}$. By induction on $n$
we shall define a strictly monotone sequence $(m_n)$ in $\mathbb{N}$, a strictly decreasing sequence $(N_n)$ of infinite subsets of
$\mathbb{N}$ (i.e. $N_n\supsetneq N_{n+1}$),
a sequence $(a_n)$ of mutually orthogonal norm-one elements in $E$, and a sequence $(u_n)$ of
mutually orthogonal compact tripotents in $E^{**}$ satisfying, for all $n\in\N$,
\begin{equation}\label{eq 0 prop isomorphic copies in dual}
m_n< \min N_n,
\end{equation}
\begin{equation}\label{eq 1 prop isomorphic copies in dual} a_n\in E_2^{**} (u_n),
\end{equation}
\begin{equation}\label{eq 2 prop isomorphic copies in dual} \left\| u_n\right\|_{\varphi_m} < r \frac{\varepsilon_n}{63} \quad \forall m\in N_{n},
\end{equation}
\begin{equation}\label{eq 3 prop isomorphic copies in dual} | \varphi_{m_n} (a_n) |> \left( 1-  r  \sum_{i=1}^{n-1} \frac{\varepsilon_i}{3} \right) r \left(1-\sum_{i=1}^{n} \varepsilon_i\right)
\end{equation}
(where $\sum_{i=1}^{0}=0$) and, for $v_n =  u_1 +\ldots +u_n$,
\begin{equation}\label{eq 4 prop isomorphic copies in dual}
\left\| \sum_{m\in N_n} \alpha_m \frac{\varphi_{m} P_0 (v_n)}{\|\varphi_{m} P_0 (v_n)\|} \right\|
\geq r \left(1-\sum_{i=1}^{n} \varepsilon_i\right) \sum_{m\in N_n} |\alpha_m| \quad\forall\alpha_m\in\mathbb{C}.
\end{equation}
Let us note that since
$$\left( 1-  r  \sum_{i=1}^{n-1} \frac{\varepsilon_i}{3} \right) r \left(1-\sum_{i=1}^{n} \varepsilon_i\right) > r \left(1-2 \sum_{i=1}^{n} \varepsilon_i\right) $$
$$> r \left(1-2 \sum_{i=1}^{\infty} \varepsilon_i\right) = r \left(1- \varepsilon\right), $$
the inequality in \eqref{eq 3 prop isomorphic copies in dual} proves $| \varphi_{m_n} (a_n) |> r \left(1- \varepsilon\right)$,
for every $n\in \mathbb{N}$.
The statement of the proposition will follow for $c_n = e^{i \theta_n} a_n$ for a suitable choice of $\theta_n\in \mathbb{R}$ such that
$ \varphi_{m_n} (c_n) =| \varphi_{m_n} (a_n) |$ for every natural $n$.\smallskip

We deal first with the case $n=1$. Set $N_0 = \mathbb{N}$. Let us take a natural number $j_1$ such that
$3 \frac{21}{\sqrt{j_1}} < r \varepsilon_1$.
Let $\delta_1 = \widetilde{\delta} (j_1, \varepsilon_1/2)>0$ be given by Lemma \ref{l compact tripotents in Cor 2.11}.
By James' distortion theorem there exist mutually disjoint finite subsets
$G^{(1)}_k\subset N_0$, finite sequences $(\lambda^{(1)}_m)_{m\in F_k^{(1)}}\subset \mathbb{C}$ such that
\begin{equation}\label{eq 1 James} \sum_{m\in G^{(1)}_k} |\lambda^{(1)}_m| \leq \frac1r, \hbox{ for every } k\in \mathbb{N},
\end{equation} and the functionals $\displaystyle \phi_k^{(1)} = \sum_{m\in G^{(1)}_k} \lambda^{(1)}_m \varphi_m$ satisfy
\begin{equation}\label{eq 2 James} \sum_{k\in N_0} |\alpha_k|
\geq \left\| \sum_{k\in N_0} \alpha_k \phi_{k}^{(1)} \right\| \geq \left(1-\delta_1 \right) \sum_{k\in N_0} |\alpha_k|
\quad\forall\alpha_k\in\mathbb{C}.
\end{equation}

By Lemma \ref{l compact tripotents in Cor 2.11} and the choice of $\delta_1$, we find mutually orthogonal elements
$a_1^{(1)},\ldots, a_{j_1}^{(1)}$ of norm one in $E$ and mutually orthogonal compact tripotents
$u_1^{(1)},\ldots, u_{j_1}^{(1)}$ in $E^{**}$ satisfying $a_k^{(1)} \in E^{**}_2 ( u_k^{(1)})$,
$$\Big\| \phi^{(1)}_k- \frac{\phi^{(1)}_k P_2 (r(a_k^{(1)}))}{\|\phi^{(1)}_k P_2 (r(a_k^{(1)}))\|}\Big\|
<\frac{\varepsilon_1}{2}, \hbox{ and } \phi^{(1)}_k (a_k^{(1)})> 1-\varepsilon_1/2,$$
for every $k=1,\ldots, j_1$.
Keeping in mind that $u_1^{(1)},\ldots, u_{j_1}^{(1)}$ are mutually orthogonal we deduce that
$$0\leq \sum_{k=1}^{j_1} \|u_k^{(1)}\|_{\varphi_m}^2 = \left\| \sum_{k=1}^{j_1} u_k^{(1)} \right\|_{\varphi_{m}}^2 \leq 1 \quad\forall m\in N_0.$$
It follows that there exist $k_1\leq j_1$ in $N_0$ and an infinite subset $N_1\subset N_0$
such that $k_1\notin N_1 $, $\max G_k^{(1)} <\min N_{1}$, for every $k=1,\ldots, j_1$, and
\begin{equation}\label{e 4.6 for n=1} \| u_{k_1}^{(1)} \|_{\varphi_{m}}^2 \leq \frac{1}{j_1} \quad\forall m\in N_1.
\end{equation}
Lemma \ref{l FerPer MathScan 2009 L21} implies that
\begin{equation}\label{eq 4.8 for n=1 step 1}\left\| \varphi_m - \varphi_m P_0 (u_{k_1}^{(1)})  \right\| \leq \frac{21}{\sqrt{j_1}}
\quad\forall m\in N_1.
\end{equation}
Define $a_1 = a_{k_1}^{(1)}$ and $u_1 = u_{k_1}^{(1)}$.
By \eqref{e 4.6 for n=1} and the choice of $j_1$ we obtain
\eqref{eq 2 prop isomorphic copies in dual} for $n=1$.\smallskip

We claim that there exists $m_1\in G_{k_1}^{(1)}$ satisfying $|\varphi_{m_1} (a_1) | > r (1-\varepsilon_1)$.
Otherwise, $|\varphi_{m} (a_1) | \leq r (1-\varepsilon_1)$, for every $m\in G_{k_1}^{(1)}$.
Recalling that $\phi^{(1)}_{k_1} (a_1)> 1-\frac{\varepsilon_1}{2}$, we have
$$1-\frac{\varepsilon_1}{2} < \phi^{(1)}_{k_1} (a_1) =  \left|  \sum_{m\in G^{(1)}_{k_1}} \lambda^{(1)}_m \varphi_m (a_1)  \right|
\leq r (1-\varepsilon_1) \sum_{m\in G^{(1)}_{k_1}} |\lambda^{(1)}_m| \leq 1-\varepsilon_1,$$ which is impossible.
This defines $m_1\notin N_1$ as in \eqref{eq 0 prop isomorphic copies in dual} and
\eqref{eq 3 prop isomorphic copies in dual} for $n=1$.\smallskip

Now, we set $\widetilde{\varphi}^{(1)}_m := \varphi_m P_0 (u_{1}).$ By \eqref{eq 4.8 for n=1 step 1} we have $\left\| \varphi_m - \widetilde{\varphi}^{(1)}_m \right\| \leq \frac{21}{\sqrt{j_1}}< \varepsilon_1/3 <1/3,$ for every $m\in N_1$.
The inequalities
$$1 = \| \varphi_m\| \geq  \| \widetilde{\varphi}^{(1)}_m\| \hbox{ and } 0\leq  1-  \| \widetilde{\varphi}^{(1)}_m\|
\leq \| \varphi_m - \widetilde{\varphi}^{(1)}_m\|<\frac{21}{\sqrt{j_1}}$$ imply that
$\| \widetilde{\varphi}^{(1)}_m\|\geq 1- \frac{21}{\sqrt{j_1}}> 1-\frac13> \frac12.$
Therefore,
$$\left\| \widetilde{\varphi}^{(1)}_m -\frac{\widetilde{\varphi}^{(1)}_m}{\|\widetilde{\varphi}^{(1)}_m\|}\right\|
= \|\widetilde{\varphi}^{(1)}_m\| \left| 1- \frac{1}{\|\widetilde{\varphi}^{(1)}_m\|} \right| \leq \frac{1}{\|\widetilde{\varphi}^{(1)}_m\|}-1 $$
$$= \frac{1}{\|\widetilde{\varphi}^{(1)}_m\|} (1-\|\widetilde{\varphi}^{(1)}_m\|) < 2 \frac{21}{\sqrt{j_1}},$$
which shows that
$$\left\| \varphi_m -\frac{\varphi_m P_0 (u_{1})}{\|\varphi_m P_0 (u_{1})\|}\right\|
= \left\| \varphi_m - \frac{\widetilde{\varphi}^{(1)}_m}{\|\widetilde{\varphi}^{(1)}_m\|} \right\| $$
$$\leq \left\| \varphi_m - \widetilde{\varphi}^{(1)}_m \right\| + \left\| \widetilde{\varphi}^{(1)}_m - \frac{\widetilde{\varphi}^{(1)}_m}{\|\widetilde{\varphi}^{(1)}_m\|} \right\| <  3 \frac{21}{\sqrt{j_1}}< r \varepsilon_1.$$
By hypothesis, $(\varphi_m)$ is a normalized sequence in $E^{*}$ spanning $\ell_1$ $r$-isomorphically, hence
$$\left\| \sum_{m\in N_1} \alpha_m \frac{\varphi_{m} P_0 (u_1)}{\|\varphi_{m} P_0 (u_1)\|} \right\|
\geq \left\| \sum_{m\in N_1} \alpha_m \varphi_{m} \right\|
-\left\| \sum_{m\in N_1} \alpha_m \left(\varphi_m -\frac{\widetilde{\varphi}^{(1)}_m}{\|\widetilde{\varphi}^{(1)}_m\|}\right) \right\|$$
$$ \geq r  \sum_{m\in N_1} |\alpha_m| - r \varepsilon_1 \sum_{m\in N_1} |\alpha_m| =r \left(1-\varepsilon_1\right) \sum_{m\in N_1} |\alpha_m|
\quad\forall\alpha_m\in\mathbb{C}.$$
This proves \eqref{eq 4 prop isomorphic copies in dual} for $n=1$, which concludes the first induction step.\smallskip

Suppose now, by the induction hypothesis, that $m_k$, $N_k$, $a_k$, and $u_k$ have been defined for $k\leq n$ according to
\eqref{eq 0 prop isomorphic copies in dual} -- \eqref{eq 4 prop isomorphic copies in dual}.
By \cite[Theorem 3.9]{FerPe10} the element $\displaystyle v_n = \sum_{k=1}^{n} u_k$
is a compact tripotent in $E^{**}$, therefore $F_{n}:= E\cap E^{**}_0 (v_n)$ is a weak$^*$-dense subtriple of $E^{**}_0 (v_n)$
whose second dual, $F_{n}^{**}$, identifies with $E^{**}_0 (v_n)$.\smallskip

To simplify notation, we write $\displaystyle\psi^{(n)}_{m} = \frac{\varphi_m P_0 (v_{n})}{\|\varphi_m P_0 (v_{n})\|}$ ($m\in N_{n}$),
and we regard $(\psi^{(n)}_m)$ as a normalized sequence in $F_{n}^{*}$. By \eqref{eq 4 prop isomorphic copies in dual}
$$\sum_{m\in N_n} |\alpha_m| \geq \left\| \sum_{m\in N_n} \alpha_m \psi^{(n)}_m \right\|
\geq r \left(1-\sum_{i=1}^{n} \varepsilon_i\right) \sum_{m\in N_n} |\alpha_m| \quad\forall\alpha_m\in\mathbb{C}$$
that is, $(\psi^{(n)}_{m})$ is a normalized basis spanning $\ell_1$
$\left(r \left(1-\sum_{i=1}^{n} \varepsilon_i\right)\right)$-isomorphically.\smallskip

Let us take a natural number $j_{n+1}$ such that $3 \frac{21}{\sqrt{j_{n+1}}} < r \varepsilon_{n+1}$. Let $\delta_{n+1} = \widetilde{\delta} (j_{n+1}, \varepsilon_{n+1}/2)>0$ given by Lemma \ref{l compact tripotents in Cor 2.11}.
By James' distortion theorem there exist mutually disjoint finite subsets $G^{(n)}_k\subset N_{n}$ $(k\in\mathbb{N})$, finite sequences
$(\lambda^{(n)}_m)_{m\in G_k^{(n)}}\subset \mathbb{C}$ such that
\begin{equation}\label{eq 1 James n+1} \displaystyle \sum_{m\in G^{(n)}_k} |\lambda^{(n)}_m|
\leq \frac{1}{\displaystyle r \left(1-\sum_{i=1}^{n} \varepsilon_i\right)} \hbox{ for all } k\in \mathbb{N},
\end{equation} and the functionals $\displaystyle \phi_k^{(n)} = \sum_{m\in G^{(n)}_k} \lambda^{(n)}_m \psi^{(n)}_{m}$ satisfy
\begin{equation}\label{eq 2 James n+1} \sum_{k\in N_n} |\alpha_k|  \geq \left\| \sum_{k\in N_n} \alpha_k \phi_{k}^{(n)} \right\|
\geq \left(1-\delta_{n+1} \right) \sum_{k\in N_n} |\alpha_k| \quad\forall\alpha_m\in\mathbb{C}.
\end{equation}

By Lemma \ref{l compact tripotents in Cor 2.11} and the choice of $\delta_{n+1}$, we find mutually orthogonal elements
$a_1^{(n)},\ldots, a_{j_{n+1}}^{(n)}$ of norm one in $F_{n}= E\cap E^{**}_0 (v_n)$ and mutually orthogonal compact tripotents
$u_1^{(n)},\ldots, u_{j_{n+1}}^{(n)}$ in $F_{n}^{**}\cong E^{**}_0 (v_n)$ satisfying $a_k^{(n)} \in E^{**}_2 (u_k^{(n)})$,
and
$$\Big\| \phi^{(n)}_k- \frac{\phi^{(n)}_k P_2 (r(a_k^{(n)}))}{\|\phi^{(n)}_k P_2 (r(a_k^{(n)}))\|}\Big\|
<\frac{\varepsilon_{n+1}}{2}, \hbox{ and } \phi^{(n)}_k (a_k^{(n)})> 1-\frac{\varepsilon_{n+1}}{2},$$
for every $k=1,\ldots, j_{n+1}$.
We remark that $v_n\perp u_k^{(n)}$, for every $k=1,\ldots, j_{n+1}$ because $u_k^{(n)}\in E^{**}_0 (v_n)$.\smallskip

Keeping in mind that $u_1^{(n)},\ldots, u_{j_{n+1}}^{(n)}$ are mutually orthogonal we deduce that
$$0\leq \sum_{k=1}^{j_{n+1}} \|u_k^{(n)}\|_{\varphi_m}^2 = \left\| \sum_{k=1}^{k_{n+1}} u_k^{(n)} \right\|_{\varphi_{m}}^2 \leq 1
\quad\forall m\in N_n.$$
It follows that there exist $k_{n+1}\leq j_{n+1}$ in $N_n$, and an infinite subset $N_{n+1}\subset N_n$
such that $k_{n+1}\notin N_{n+1}$, $\max G_k^{(n)} <\min N_{n+1}$, for $k=1,\ldots,j_{n+1}$, and
\begin{equation}\label{e 4.6 for n+1} \| u_{k_{n+1}}^{(n)} \|_{\varphi_{m}}^2 \leq \frac{1}{j_{n+1}}, \hbox{ for all } m\in N_{n+1}.
\end{equation}
Lemma \ref{l FerPer MathScan 2009 L21} implies that
\begin{equation}\label{eq 4.8 for n+1 step 1}\left\| \varphi_m - \varphi_m P_0 (u_{k_{n+1}}^{(n)})  \right\| \leq \frac{21}{\sqrt{j_{n+1}}}
\quad\forall m\in N_{n+1}.
\end{equation}
Define $a_{n+1} = a_{k_{n+1}}^{(n)}$ and $u_{n+1} = u_{k_{n+1}}^{(n)}$.
By \eqref{e 4.6 for n+1} and the choice of $j_{n+1}$ we obtain
\eqref{eq 2 prop isomorphic copies in dual} for $n+1$.
Since $a_{n+1}\in F_{n}= E\cap E^{**}_0 (v_n)$, and $a_1,\ldots,a_n\in E_2^{**} (v_n)$, it follows that $a_{n+1}\perp a_k,$ for $k=1,\ldots,n$.
Therefore, $a_1,\ldots, a_{n+1}$ are mutually orthogonal elements in the closed unit ball of $E$.\smallskip

We claim that there exists $m_{n+1}\in G_{k_{n+1}}^{(n)}$ satisfying
\begin{equation}\label{e claim one}\frac{1}{\|\varphi_{m_{n+1}} P_0 (v_{n})\|} |\varphi_{m_{n+1}} (a_{n+1}) |
> r \left(1-\sum_{i=1}^{n+1} \varepsilon_i\right).
\end{equation}
Otherwise, $\displaystyle \frac{1}{\|\varphi_{m} P_0 (v_{n})\|} |\varphi_{m} (a_{n+1}) |
\leq r \left(1-\sum_{i=1}^{n+1} \varepsilon_i\right)$, for all $m\in G_{k_{n+1}}^{(n)}$.
Recalling that $\phi^{(n)}_{k_{n+1}} (a_{n+1})> 1-\frac{\varepsilon_{n+1}}{2}$, and
$a_{n+1}\in F_{n}= E\cap E^{**}_0 (v_n)$, we have
$$\psi^{(n)}_{m} (a_{n+1}) = \frac{\varphi_m P_0 (v_{n})}{\|\varphi_m P_0 (v_{n})\|} (a_{n+1})
= \frac{1}{\|\varphi_m P_0 (v_{n})\|} \varphi_m (a_{n+1}),$$
which gives  $$1-\frac{\varepsilon_{n+1}}{2} < \phi^{(n)}_{k_{n+1}} (a_{n+1})
=  \left|  \sum_{m\in G^{(n)}_{k_{n+1}}} \lambda^{(n)}_m \psi^{(n)}_{m} (a_{n+1})  \right|$$
$$ \leq r \left(1-\sum_{i=1}^{n+1} \varepsilon_i\right) \sum_{m\in G^{(n)}_{k_{n+1}}} |\lambda^{(n)}_m|
\leq r \left(1-\sum_{i=1}^{n+1} \varepsilon_i\right) \frac{1}{\displaystyle r \left(1-\sum_{i=1}^{n} \varepsilon_i\right)} $$
$$=  1-\frac{\varepsilon_{n+1}}{\displaystyle \left(1-\sum_{i=1}^{n} \varepsilon_i\right)}<1-\frac{\varepsilon_{n+1}}{2}$$
which is impossible. This proves the claim in \eqref{e claim one}.\smallskip

By \eqref{eq 2 prop isomorphic copies in dual} for $i\leq n,$ we have
$$\|v_n\|_{\varphi_m} = \left\|\sum_{i=1}^{n} u_i\right\|_{\varphi_m}\leq \sum_{i=1}^{n} \|u_i\|_{\varphi_m}
<r \sum_{i=1}^{n} \frac{\varepsilon_i}{63} \quad\forall m\in N_{n+1}.$$
Now, since $1=\left\|  \varphi_{m_{n+1}} \right\|\geq \left\|  \varphi_{m_{n+1}} P_0 (v_n)\right\|$
we deduce via  Lemma \ref{l FerPer MathScan 2009 L21} that
$$0\leq 1-\left\|  \varphi_{m_{n+1}} P_0 (v_n)\right\| \leq \left\| \varphi_{m_{n+1}} -  \varphi_{m_{n+1}} P_0 (v_n)\right\|
< 21 r  \sum_{i=1}^{n} \frac{\varepsilon_i}{63}$$
which implies that $$ 1- r  \sum_{i=1}^{n} \frac{\varepsilon_i}{3} < \left\|  \varphi_{m_{n+1}} P_0 (v_n)\right\|$$
hence
$$|\varphi_{m_{n+1}} (a_{n+1}) |
> \left( 1- r  \sum_{i=1}^{n} \frac{\varepsilon_i}{3} \right) r \left(1-\sum_{i=1}^{n+1} \varepsilon_i\right)$$
by \eqref{e claim one}.
We have thus defined $m_{n+1}\notin N_{n+1}$ such that $m_n<m_{n+1}$ and such that \eqref{eq 3 prop isomorphic copies in dual}
holds for $n+1$.\smallskip

Finally we show \eqref{eq 4 prop isomorphic copies in dual} for $n+1$.
Since \eqref{eq 2 prop isomorphic copies in dual} holds for $i\leq n+1$ we get
\begin{equation}\label{eq 4.8 for n+1 step 2} \|v_{n+1}\|_{\varphi_{m}} = \left\|\sum_{i=1}^{n+1} u_{i} \right\|_{\varphi_{m}}
\leq \sum_{i=1}^{n+1} \left\| u_{i} \right\|_{\varphi_{m}} < r \sum_{i=1}^{n+1} \frac{\varepsilon_i}{63} \quad\forall m\in N_{n+1}.
\end{equation}
To simplify notation, let us denote $\widetilde{\varphi}^{(n+1)}_m := \varphi_m P_0 (v_{n+1})$, $m\in N_{n+1}$.
Lemma \ref{l FerPer MathScan 2009 L21} and \eqref{eq 4.8 for n+1 step 2} imply that
$$\left\| \varphi_m - \widetilde{\varphi}^{(n+1)}_m \right\| < 21 r \sum_{i=1}^{n+1} \frac{\varepsilon_i}{63}
= r \sum_{i=1}^{n+1} \frac{\varepsilon_i}{3} < r \frac{\varepsilon}{6} < \frac16 \quad\forall m\in N_{n+1}.$$
The inequalities
$$1 = \| \varphi_m\| \geq  \| \widetilde{\varphi}^{(n+1)}_m\| \hbox{ and }
0\leq  1-  \| \widetilde{\varphi}^{(n+1)}_m\|\leq \| \varphi_m - \widetilde{\varphi}^{(n+1)}_m\|<r \sum_{i=1}^{n+1} \frac{\varepsilon_i}{3},$$
imply that $\displaystyle\| \widetilde{\varphi}^{(n+1)}_m\|\geq 1- r \sum_{i=1}^{n+1} \frac{\varepsilon_i}{3}>1-\frac16> \frac12.$ Therefore, $$\left\| \widetilde{\varphi}^{(n+1)}_m -\frac{\widetilde{\varphi}^{(n+1)}_m}{\|\widetilde{\varphi}^{(n+1)}_m\|}\right\| = \|\widetilde{\varphi}^{(n+1)}_m\| \left| 1- \frac{1}{\|\widetilde{\varphi}^{(n+1)}_m\|} \right| \leq \frac{1}{\|\widetilde{\varphi}^{(n+1)}_m\|}-1 $$
$$= \frac{1}{\|\widetilde{\varphi}^{(n+1)}_m\|} (1-\|\widetilde{\varphi}^{(n+1)}_m\|) < 2 r \sum_{i=1}^{n+1} \frac{\varepsilon_i}{3}$$
which shows that
$$\left\| \varphi_m -\frac{\varphi_m P_0 (v_{n+1})}{\|\varphi_m P_0 (v_{n+1})\|}\right\|
= \left\| \varphi_m - \frac{\widetilde{\varphi}^{(n+1)}_m}{\|\widetilde{\varphi}^{(n+1)}_m\|} \right\|$$
$$ \leq \left\| \varphi_m - \widetilde{\varphi}^{(n+1)}_m \right\| + \left\| \widetilde{\varphi}^{(n+1)}_m
- \frac{\widetilde{\varphi}^{(n+1)}_m}{\|\widetilde{\varphi}^{(n+1)}_m\|} \right\| <  3 r \sum_{i=1}^{n+1} \frac{\varepsilon_i}{3}
= r \sum_{i=1}^{n+1} {\varepsilon_i},$$
for all $m\in N_{n+1}$.
By hypothesis, $(\varphi_m)$ is a normalized sequence in $E^{*}$ spanning $\ell_1$ $r$-isomorphically, hence
$$\left\| \sum_{m\in N_{n+1}} \alpha_m \frac{\varphi_{m} P_0 (v_{n+1})}{\|\varphi_{m} P_0 (v_{n+1})\|} \right\|
\geq $$ $$\left\| \sum_{m\in N_{n+1}} \alpha_m \varphi_{m} \right\| -\left\| \sum_{m\in N_{n+1}} \alpha_m \left(\varphi_m
-\frac{\varphi_m P_0 (v_{n+1})}{\|\varphi_m P_0 (v_{n+1})\|}\right) \right\|$$
$$ \geq r  \sum_{m\in N_{n+1}} |\alpha_m| - r \left(\sum_{i=1}^{n+1} \varepsilon_i\right) \sum_{m\in N_{n+1}} |\alpha_m|
=r \left(1- \sum_{i=1}^{n+1} \varepsilon_i \right) \sum_{m\in N_{n+1}} |\alpha_m|$$
for all $\alpha_m\in\mathbb{C}$.
This proves \eqref{eq 4 prop isomorphic copies in dual} for $n+1$ and shows \eqref{eq gl4.2}.\smallskip

By an extraction lemma of Simons \cite{Sim-DPP} we may (after passing to appropriate subsequences of $(\varphi_{m_n})$ and $(c_n)$
which we still denote by  $(\varphi_{m_n})$ and $(c_n)$) suppose that $\displaystyle \sum_{k\neq n}\betr{\varphi_{m_n}(c_k)}<\eps'$ for all $n$
where $\eps'>0$ is such that $r(1-\eps)-\eps'>r(1-2\eps)$. That the subtriple $\mathcal{C}$ generated by the $c_n$ is isometric to a commutative C$^*$-algebra can be seen as in \cite[Th.\ 2.3$b)\Rightarrow b')$]{FerPe09}.
Fix $(\alpha_n)\in\ell_1$, choose $\theta_n\in\mathbb{C}$ such that $\theta_n\alpha_n=\betr{\alpha_n}$ and set $\displaystyle c=\sum_{k\geq 1}\theta_k c_k$.
Then $\norm{c}\leq1$ and
\bglst
\left\|\sum_{n\geq 1}\alpha_n{\varphi_{m_n}}_{|\mathcal{C}}\right\|
&\geq&\left|\sum_{n\geq 1} \alpha_n\varphi_{m_n}(c)\right| \\ &=& \left|\sum_n\betr{\alpha_n}\varphi_{m_n}(c_n)+\sum_n\alpha_n \left(\sum_{k\neq n}\varphi_{m_n}(\theta_kc_k)\right)\right|\\
&\geq&r(1-\eps) \sum_n\betr{\alpha_n} -\sum_n\betr{\alpha_n}\left(\sum_{k\neq n}\betr{\varphi_{m_n}(c_k)}\right)\\
&\geq&(r(1-\eps)-\eps') \sum_n\betr{\alpha_n}\geq r(1-2\eps)\sum_n\betr{\alpha_n}.
\eglst
Up to an adjustment of $\eps$ this ends the proof.
\end{proof}


\end{document}